\documentclass{amsart}

\usepackage{mathrsfs,eucal}
\usepackage[dvips]{graphicx}
\usepackage{pstricks}

\usepackage{amsthm,amssymb,amsmath}

\newtheorem{thm}{Theorem}[section]
\newtheorem{cor}[thm]{Corollary}
\newtheorem{lem}[thm]{Lemma}
\newtheorem{prop}[thm]{Proposition}

\newtheorem{defn}[thm]{Definition}
\newtheorem{rem}[thm]{Remark}
\newtheorem*{aspta}{Assumption A}
\newtheorem{example}[thm]{Example}
\numberwithin{equation}{section}

\newcommand{\argmax}{\mathop{\mathrm{argmax}}}
\newcommand{\pen}{\mathrm{pen}}

\newcommand{\cl}{\mathop{\mathrm{cl}}}
\newcommand{\vvverto}{\mathopen{|\hspace{-0.12em}|\hspace{-0.12em}|}}
\newcommand{\vvvertc}{\mathclose{|\hspace{-0.12em}|\hspace{-0.12em}|}}
\newcommand{\tnorm}[1]{\vvverto #1 \vvvertc}

\renewcommand{\limsup}{\mathop{\overline{\mathrm{lim}}}}
\renewcommand{\liminf}{\mathop{\underline{\mathrm{lim}}}}

\begin{document}

\title{Consistent order estimation and minimal penalties}

\author{Elisabeth Gassiat}
\address{Laboratoire de Math{\'e}matiques, 
Universit{\'e} Paris-Sud, 
B{\^a}timent 425, 
91405 Orsay Cedex, France}
\email{elisabeth.gassiat@math.u-psud.fr}

\author{Ramon van Handel}
\address{Sherrerd Hall, Room 227,
Princeton University, 
Princeton, NJ 08544, USA.}
\email{rvan@princeton.edu}

\begin{abstract}
Consider an i.i.d.\ sequence of random variables whose distribution 
$f^\star$ lies in one of a nested family of models $\mathcal{M}_q$, 
$q\ge 1$.  The smallest index $q^\star$ such that 
$\mathcal{M}_{q^\star}$ contains $f^\star$ is called the model order. We 
establish strong consistency of the penalized likelihood order estimator 
in a general setting with penalties of order $\eta(q)\log\log n$, where 
$\eta(q)$ is a dimensional quantity.  Moreover, such penalties are shown 
to be minimal.  In contrast to previous work, an a priori upper bound on 
the model order is not assumed. The results rely on a sharp 
characterization of the pathwise fluctuations of the generalized 
likelihood ratio statistic under entropy assumptions on the model 
classes.  Our results are applied to the geometrically complex problem
of location mixture order estimation, which is widely used but 
poorly understood.
\end{abstract}


\maketitle

\section{Introduction}
\label{sec:intro}

Let $(X_k)_{k\ge 1}$ be a sequence of random variables whose 
distribution $f^\star$ lies in one of a nested family of models 
$(\mathcal{M}_q)_{q\ge 1}$, indexed (and ordered) by the integers. We 
define the model order as the smallest index $q^\star$ such that the 
true distribution $f^\star$ lies in the corresponding model class. The 
model order typically determines the most parsimonious representation of 
the true distribution of the underlying model (for example, it might 
determine the parametrization of the model which has the smallest 
possible dimension).  On the other hand, the model order often has a 
concrete interpretation in terms of the modelling of the underlying 
phenomenon (for example, the estimation of the number of clusters in a 
data set, or the number of regimes in an economic time series).  
Therefore, the problem of estimating the model order from observed data 
is of significant practical, as well as theoretical, interest.

Of course, a satisfactory solution to this problem must provide an 
estimation method that does not assume prior knowledge on the 
unknown distribution $f^\star$. In particular, prior bounds on model 
order and on parameter sets should be avoided.  Yet, in this light, even 
one of the most widely used model selection criteria---the Bayesian 
Information Criterion (BIC) of Schwarz---is poorly understood.  The 
chief motivation for the use of BIC (as opposed to other model selection 
criteria, such as Akaike's Information Criterion) is that it is expected 
to yield a strongly consistent estimator of the model order.  However, 
almost all existing consistency proofs assume a prior upper bound on the 
order as well as compactness of the parameter set. As is 
emphasized by Csisz{\'a}r and Shields \cite{CS00}, this is hardly 
satisfactory from the theoretical point of view, and provides little 
confidence in the basic motivation for this method.  More delicate 
questions, such as the minimal penalty that yields a consistent order 
estimator in absence of a prior bound on the order, remain open 
(the problem of identifying the minimal penalty, which minimizes the 
probability of underestimating the order, is also raised in 
\cite{CS00}).

In this paper we consider a general class of penalized likelihood
order estimators of the form
$$
        \hat q_n = \argmax_{q\ge 1}\left\{
        \sup_{f\in\mathcal{M}_q}\ell_n(f) - \pen(n,q)\right\},
$$
where $\pen(n,q)$ is a penalty function and $\ell_n(f)$ is the 
likelihood of $(X_k)_{1\le k\le n}$ under the distribution $f$. Our aim is to 
understand what penalties yield strong consistency of the order 
estimator, that is, $\hat q_n\to q^\star$ as $n\to\infty$ a.s. 
Characterizing strong consistency hinges on a precise understanding of 
the pathwise fluctuations of the likelihood ratio statistic
$$
        \sup_{f\in\mathcal{M}_q}\ell_n(f) -
        \sup_{f\in\mathcal{M}_{q^\star}}\ell_n(f),
$$
as $n\to\infty$, \emph{uniformly} in the model order $q>q^\star$.
When there is a known upper bound on the order 
$q^\star\le q_{\rm max}<\infty$ and the model classes $\mathcal{M}_q$ 
are regularly parametrized by a compact subset of Euclidean space, an 
upper bound on the pathwise fluctuations can be obtained by classical 
parametric methods: Taylor expansion of the likelihood and an application 
of a law of iterated logarithm.  This approach forms the basis for most 
consistency proofs for penalized likelihood order estimators in the 
literature, for example 
\cite{HQ79,Nis88,Fin90,Ker00,Cha06}.  However, such techniques fail in 
the absence of a prior upper bound: even though each model class 
$\mathcal{M}_q$ is finite dimensional, the full model 
$\mathcal{M}=\bigcup_q\mathcal{M}_q$ is infinite dimensional and, as 
such, the problem in the absence of a prior upper bound is inherently 
nonparametric. 
When the classes $\mathcal{M}_q$ are noncompact
one must introduce sieves $\mathcal{M}_q^n 
\subset\mathcal{M}_q^{n+1}\subset\cdots\subset\mathcal{M}_q$,
complicating the problem further (in this case even the parametric
theory remains poorly understood \cite{Har85,BC93,LS04}).  
An entirely different approach based on 
universal coding theory \cite{Fin90,Kie93,LN94,GB03,CMR05,CGG09} yields 
bounds on the pathwise fluctuations that do not 
require prior bounds on the order or compactness of the models.  
However, these bounds are far from tight and cannot even establish 
consistency of BIC, let alone smaller penalties (this appears to be a 
fundamental limitation of this approach due to Rissanen's theorem, see 
\cite{Ris86,BRY99}).

The problem area that is investigated in this paper was initiated in the 
work of Csisz{\'a}r and Shields \cite{CS00,Csis02}, who proved 
consistency of BIC for Markov chain order estimation in absence of a 
prior bound on the order (see also \cite{CsisTala}).  To our knowledge, 
little progress has been made on this subject beyond their work.  The 
proofs in \cite{CS00,Csis02} rely heavily on the availability of an 
explicit expression for the maximum likelihood for Markov chains, and 
employ delicate estimates specific to that setting. Their techniques are 
therefore not well suited to investigating such problems in other 
settings.  Moreover, the methods of \cite{CS00,Csis02} do not yield 
minimal penalties.  However, the Markov chain case was recently 
reconsidered in \cite{Han09} using very different techniques based 
on empirical process theory, which are potentially much more generally 
applicable and which shed light on minimal penalties.

The main results of this paper provide generally applicable upper and 
lower bounds on the pathwise fluctuations of the likelihood ratio 
statistic uniformly in the model order $q>q^\star$, for the case of 
i.i.d.\ observations $(X_k)_{k\ge 1}$, without a prior bound on the 
model order and in possibly noncompact parameter spaces.  These results 
are then used to investigate strong consistency of penalized likelihood 
order estimators. We use empirical process methods as in \cite{Han09}, 
but the difficulties to be surmounted in the present setting are of a 
different nature. The main difficulty for Markov chain models in 
\cite{CS00,Csis02,Han09} is their dependence structure;
in the present paper we assume i.i.d.\ models.  On the other hand, 
the geometric structure of Markov chains is exceedingly simple: the 
family of $q$th-order Markov chains in the Hellinger distance is simply 
a Euclidean ball when viewed in the appropriate parametrization. In 
contrast, in general order estimation problems, one is often faced with 
model classes that are geometrically very complex. An important case 
study that will be considered in this paper are location mixture models 
(widely used in practice for clustering), which possess a notoriously 
complicated non-regular geometry.  We will be able, for example, to 
establish strong consistency of BIC for mixture order estimation in
absence of a prior bound on the order or on the parameter set, providing
a counterpart to the results of Csisz{\'a}r and Shields \cite{CS00} in
a setting very different than that of Markov chains.

The techniques developed here originate in our attempts to understand 
the order estimation problem for hidden Markov models (HMM) \cite{GB03}. 
In that setting, consistency of BIC (even with a prior bound on the 
order) remains unknown.  The two cases considered here and in 
\cite{Han09}---Markov chains and i.i.d.\ mixtures---can be viewed as two 
extreme cases of HMM.  While our approach provides a substantial step 
towards understanding the HMM setting, a striking and as of yet poorly 
understood breakdown in the ergodicity of HMM \cite{GK00} has so far 
impeded further progress in this direction.

The remainder of this paper is organized as follows.  Section 
\ref{sec:pathwise} introduces the general model under consideration, and 
states our results on the pathwise fluctuations of the likelihood ratio 
statistic.  Section \ref{sec:consist} derives the consequences for 
order estimation, and considers also the special case of location
mixture models. Proofs are given in the appendices.

\section{Pathwise fluctuations of the likelihood}
\label{sec:pathwise}

\subsection{Basic setting and notation}
\label{sec:setting}

Let $(E,\mathcal{E},\mu)$ be a measure space.  For each $q,n\ge 1$, let 
$\mathcal{M}_q^n$ be a given family of strictly positive probability 
densities with respect to $\mu$ (that is, we assume that $\int f d\mu=1$ 
and that $f>0$ $\mu$-a.e.\ for every $f\in\mathcal{M}_q^n$).  Moreover, 
we assume that $(\mathcal{M}_q^n)_{q,n\ge 1}$ is a nested family of 
models in the sense that $\mathcal{M}_q^n\subseteq\mathcal{M}_{q+1}^n$ 
and $\mathcal{M}_q^n\subseteq\mathcal{M}_q^{n+1}$ for all $q,n\ge 1$. 
Let $\mathcal{M}_q = \bigcup_{n}\mathcal{M}_q^n$, 
$\mathcal{M}^n = \bigcup_{q}\mathcal{M}_q^n$, 
$\mathcal{M}=\bigcup_{q,n}\mathcal{M}_q^n$.

Consider an i.i.d.\ sequence of $E$-valued random variables $(X_k)_{k\ge 
1}$ whose common distribution under the measure $\mathbf{P}^\star$ is 
$f^\star d\mu$, where $f^\star\in 
\mathcal{M}_{q^\star}\backslash\cl\mathcal{M}_{q^\star-1}$ for some 
$q^\star\ge 1$ (here $\cl\mathcal{M}_{q}$ denotes the 
$L^1(d\mu)$-closure of $\mathcal{M}_q$).  The index $q^\star$ is called 
the \emph{model order}.  Let us define 
$$
        \ell_n(f) = \sum_{i=1}^n \log f(X_i),\qquad\quad
        f\in\mathcal{M}.
$$
Evidently $\ell_n(f)$ is the log-likelihood of the i.i.d.\ sequence 
$(X_k)_{k\le n}$ when $X_k\sim f d\mu$.    Our aim is to study the 
pathwise fluctuations of the likelihood ratio statistic
$$
        \sup_{f\in\mathcal{M}_q^n}\ell_n(f) -
        \sup_{f\in\mathcal{M}_{q^\star}^n}\ell_n(f)
$$
as $n\to\infty$, \emph{uniformly} over the order parameter $q\ge q^\star$. 
Pathwise upper and lower bounds on the likelihood ratio statistic are the 
key ingredient in the study of strong consistency of penalized likelihood 
order estimators (see section \ref{sec:consist}).

\begin{example}[Location mixtures]
\label{ex:locmix}
The guiding example for our theory, the case of location mixtures,
will be studied in detail in section 
\ref{sec:mixorder} below.  We presently introduce this example in order
to clarify our basic setup.

Let $E=\mathbb{R}^d$ (with its Borel $\sigma$-field $\mathcal{E}$)
and let $\mu$ be the Lebesgue measure on $\mathbb{R}^d$.  We fix a
strictly positive probability density $f_0$ with respect to $\mu$,
and define $f_\theta(x) = f_0(x-\theta)$ for $x,\theta\in\mathbb{R}^d$.
Fix a sequence $T(n)\uparrow\infty$ and define
$$
	\mathcal{M}_q^n = \left\{
	\sum_{i=1}^q\pi_if_{\theta_i}:
	\pi_i\ge 0,~
	\sum_{i=1}^q \pi_i = 1,~
	\|\theta_i\|\le T(n)
	\right\}.
$$
Then $\mathcal{M}_q$ is the family of all $q$-component mixtures of 
translates of the density $f_0$, while $\mathcal{M}_q^n$ is the subset
of the mixtures $\mathcal{M}_q$ whose translation parameters 
$(\theta_i)_{i=1,\ldots,q}$ are restricted to a ball of radius $T(n)$.
The number of components $q^\star$ of the true mixture $f^\star\in\mathcal{M}$ 
can be estimated from observations using the order estimator
$$
        \hat q_n = \argmax_{q\ge 1}\left\{
        \sup_{f\in\mathcal{M}_q^n}\ell_n(f) - \pen(n,q)\right\}.
$$
Pathwise control of the likelihood ratio statistic allows us to
identify what penalties $\pen(n,q)$ and cutoff sequences $T(n)$ 
yield strong consistency of $\hat q_n$
(cf.\ section \ref{sec:mixorder}).
\end{example}

\begin{rem}
\label{rem:esssup}
To avoid measurability problems and other technical complications, we 
employ throughout this paper the simplifying convention that all uncountable 
suprema (such as $\sup_{f\in\mathcal{M}_q^n}\ell_n(f)$) are interpreted as 
essential suprema with respect to the measure $\mathbf{P}^\star$.  In 
the majority of applications the model classes $\mathcal{M}_q^n$ will be 
separable, in which case the supremum and essential supremum coincide.
\end{rem}

In the sequel, we will denote by $\|\cdot\|_p$ the
$L^p(f^\star d\mu)$-norm, that is, $\|g\|_p^p = \int |g(x)|^p
f^\star(x)\mu(dx)$, and we denote by $\langle 
f,g\rangle=\int f(x)g(x)f^\star(x)\mu(dx)$ the Hilbert space inner 
product in $L^2(f^\star d\mu)$.  Define the Hellinger distance
$$
        h(f,g)^2 = \int (\sqrt{f}-\sqrt{g})^2d\mu,\qquad
        f,g\in\mathcal{M}.
$$ 
It is easily seen that $h(f,f^\star) = \|\sqrt{f/f^\star}-1\|_2$. Finally,
we will denote by $\mathcal{N}(\mathcal{Q},\delta)$ for any class of functions 
$\mathcal{Q}$ and $\delta>0$ the minimal number of brackets of 
$L^2(f^\star d\mu)$-width $\delta$ needed to cover $\mathcal{Q}$: that 
is, $\mathcal{N}(\mathcal{Q},\delta)$ is the smallest cardinality $N$ of 
a collection of pairs of functions $\{g_i^L,g_i^U\}_{i=1,\ldots, N}$ 
such that $\max_{i\le N}\|g_i^U-g_i^L\|_2\le\delta$ and for every 
$g\in\mathcal{Q}$ we have $g_i^L\le g\le g_i^U$ pointwise for some $i\le N$.

\subsection{Upper bound}

We aim to obtain a pathwise upper bound on the likelihood ratio 
statistic that holds \emph{uniformly} in $q>q^\star$. To this end, 
define for $q,n\ge 1$ and $\varepsilon>0$ the Hellinger ball
$$
	\mathcal{H}_q^n(\varepsilon)=\{\sqrt{f/f^\star}:
	f\in\mathcal{M}_q^n,~h(f,f^\star)\le\varepsilon\}.
$$
Note that the definition of $\mathcal{H}_q^n(\varepsilon)$ depends on 
$f^\star$ (which is fixed throughout the paper).  The following result 
shows that the geometry of the Hellinger balls 
$\mathcal{H}_q^n(\varepsilon)$ controls the pathwise fluctuations of 
the likelihood ratio statistic.

\begin{thm}
\label{thm:upperlil}
Suppose that for all $n$ sufficiently large
$$
        \mathcal{N}(\mathcal{H}_q^n(\varepsilon),\delta)
        \le
        \left(\frac{K(n)\varepsilon}{\delta}\right)^{\eta(q)}
$$
for all $q\ge q^\star$ and $\delta\le\varepsilon$,
where $K(n)\ge 1$ and $\eta(q)\ge q$ are increasing functions.  Then
$$
        \limsup_{n\to\infty}\frac{1}{\log K(2n)\vee\log\log n} 
        \sup_{q\ge q^\star} \frac{1}{\eta(q)}
        \left\{
        \sup_{f\in\mathcal{M}_q^n}\ell_n(f) -
        \sup_{f\in\mathcal{M}_{q^\star}^n}\ell_n(f)
        \right\} \le C
$$
$\mathbf{P}^\star$-a.s., where $C>0$ is a universal constant.
\end{thm}

The proof of Theorem \ref{thm:upperlil}
is given in Appendix \ref{sec:proofupperlil}.

The assumption of Theorem \ref{thm:upperlil} on the entropy of the 
Hellinger balls $\mathcal{H}_q^n(\varepsilon)$ states, roughly speaking, 
that the class of densities $\mathcal{M}_q^n$ endowed with the Hellinger 
distance has the same metric structure as a Euclidean ball of dimension 
$\eta(q)$ and radius of order $K(n)$, at least locally in a neighborhood 
of the true density $f^\star$.  The effective dimension $\eta(q)$ 
controls the fluctuations of the likelihood ratio statistic as a 
function of the model order, while the effective radius $K(n)$ controls 
the fluctuations as a function of time up to a minimal rate of order 
$\log\log n$.  In the following section we will see that the 
minimal $\log\log n$ rate is indeed optimal.

\begin{rem}
A bound on $\mathcal{N}(\mathcal{H}_q^n(\varepsilon),\delta)$ of the
form required by Theorem \ref{thm:upperlil} is easily obtained if
$\mathcal{M}_q^n$ are regularly parametrized classes.  That is, suppose
that we can write
$$
	\mathcal{M}_q^n = \{f_\theta : \theta\in\Theta_q^n\},\qquad
	\Theta_q^n\subset\mathbb{R}^{\eta(q)},
$$
where we have a pointwise Lipschitz estimate of the form
$$
	\textstyle{|\sqrt{f_\theta(x)/f^\star(x)}
	-\sqrt{f_{\theta'}(x)/f^\star(x)}|}
	\le F(x)\, \tnorm{\theta-\theta'}
$$
for some function $F$ in $L^2$ and norm $\tnorm{\cdot}$ on 
$\mathbb{R}^{\eta(q)}$, and
$$
	h(f_\theta,f^\star) \ge c\,\tnorm{\theta-\theta^\star}
$$
with $c>0$.   Then the requisite bound 
on $\mathcal{N}(\mathcal{H}_q^n(\varepsilon),\delta)$ follows easily
(cf.\ \cite[Example 19.7]{vdV98}).  This covers many cases of practical
interest.  However, geometrically complex models such as finite
mixtures do not admit a regular parametrization, while our results are
nonetheless sufficiently general to apply to such models
(section \ref{sec:mixorder}).  In non-regular models the entropy bound
required by Theorem \ref{thm:upperlil} is far from obvious, and the
requisite geometric analysis is of independent interest.  Such problems
are investigated by the authors in \cite{GasHanGeo}, and form the basis
for the results in section \ref{sec:mixorder} below.
\end{rem}

\subsection{Lower bound}
\label{sec:lowerlil}

Throughout this section, we specialize to the case that 
$\mathcal{M}_q^n=\mathcal{M}_q$ does not depend on $n$ (this implies
essentially that $\mathcal{M}_q$ is compact). In this setting, 
Theorem \ref{thm:upperlil} yields an upper bound of order $\log\log n$ 
on the pathwise fluctuations of the likelihood ratio statistic. The aim 
of this section is to obtain a matching lower bound of order $\log\log 
n$, which shows that the minimal rate in Theorem \ref{thm:upperlil} is 
essentially optimal. For the purposes of a lower bound uniformity 
in $q$ is irrelevant, so it suffices to restrict attention to some 
fixed $q>q^\star$.  We will in fact obtain a much stronger result in 
this case that completely characterizes the pathwise 
asymptotics of the likelihood ratio statistic for fixed $q$ in 
sufficiently smooth families.

The geometric structure required in the present section is somewhat 
different than that of Theorem \ref{thm:upperlil}.  Instead of Hellinger
balls, we consider the classes of weighted densities 
$\mathcal{D}_q = \{d_f:f\in\mathcal{M}_q,~f\ne f^\star\}$
and $\mathcal{D}=\bigcup_q\mathcal{D}_q$, where
$$
	d_f = \frac{\sqrt{f/f^\star}-1}{h(f,f^\star)},~~~
	f\in\mathcal{M},\quad f\ne f^\star.
$$
Define for $\varepsilon>0$ and $q\ge 1$ the
local weighted classes
\begin{align*}
	\mathcal{D}_q(\varepsilon)&=\{d_f:f\in\mathcal{M}_q,~
	0<h(f,f^\star)\le\varepsilon\}, \\
	\mathcal{\bar D}_q&=\bigcap_{\varepsilon>0}
	\cl\mathcal{D}_q(\varepsilon),
\end{align*}
where the closure $\cl\mathcal{D}_q(\varepsilon)$ is in $L^2(f^\star d\mu)$.
Clearly $\mathcal{\bar D}_q$ is the set of all possible limit points of
$d_f$ as $h(f,f^\star)\to 0$ in $\mathcal{M}_q$.  If the neighborhoods of
$\mathcal{\bar D}_q$ are sufficiently rich, such limits can be taken along
a continuous path in the following sense.

\begin{defn}
A point $d\in\mathcal{\bar D}_q$ is called \emph{continuously accessible} 
if there is a path 
$(f_t)_{t\in\mbox{}]0,1]}\subset\mathcal{M}_q\backslash\{f^\star\}$ such 
that the map $t\mapsto h(f_t,f^\star)$ is continuous, $h(f_t,f^\star)\to 
0$ as $t\to 0$, and $d_{f_t}\to d$ in $L^2(f^\star d\mu)$ as $t\to 0$.
The subset of all continuously accessible points in $\mathcal{\bar D}_q$ 
is denoted as $\mathcal{\bar D}_q^c$.
\end{defn}

We can now formulate the main result of this section.

\begin{thm}
\label{thm:lowerlil}
Let $q^\star\le p<q$.  Assume that
$$
        \int_0^1 \sqrt{\log\mathcal{N}(\mathcal{D}_q,u)}\,du<\infty,
$$
and that $|d|\le D$ for all $d\in\mathcal{D}_q$ with
$D\in L^{2+\alpha}(f^\star d\mu)$ for some $\alpha>0$.  Then
\begin{multline*}
        \limsup_{n\to\infty}\frac{1}{\log\log n}\left\{
        \sup_{f\in\mathcal{M}_q}\ell_n(f)-
        \sup_{f\in\mathcal{M}_p}\ell_n(f)\right\} \ge \mbox{}\\
        \sup_{g\in L^2_0(f^\star d\mu)}\left\{
        \sup_{f\in\mathcal{\bar D}_q^c}(\langle f,g\rangle)_+^2
        -\sup_{f\in\mathcal{\bar D}_p}(\langle f,g\rangle)_+^2
        \right\}\quad
        \mathbf{P}^\star\mbox{\rm-a.s.},
\end{multline*}
as well as
\begin{multline*}
        \limsup_{n\to\infty}\frac{1}{\log\log n}\left\{
        \sup_{f\in\mathcal{M}_q}\ell_n(f)-
        \sup_{f\in\mathcal{M}_p}\ell_n(f)\right\} \le \mbox{}\\
        \sup_{g\in L^2_0(f^\star d\mu)}\left\{
        \sup_{f\in\mathcal{\bar D}_q}(\langle f,g\rangle)_+^2
        -\sup_{f\in\mathcal{\bar D}_p^c}(\langle f,g\rangle)_+^2
        \right\}\quad
        \mathbf{P}^\star\mbox{\rm-a.s.},
\end{multline*}
where $L^2_0(f^\star d\mu) = \{g\in L^2(f^\star d\mu):\|g\|_2\le 1,~
\langle 1,g\rangle=0\}$.  
\end{thm}

Only the first (lower bound) part of the theorem is needed
to conclude optimality of the minimal $\log\log n$ rate in
Theorem \ref{thm:upperlil}.  Indeed, we will obtain as a corollary
the following lower bound counterpart to Theorem \ref{thm:upperlil}.

\begin{cor}
\label{cor:lowerlil}
Suppose there exists $q>q^\star$ such that the following hold.
\begin{enumerate}
\item
There is an envelope function $D:E\to\mathbb{R}$ such that
$|d|\le D$ for all $d\in\mathcal{D}_q$ and
$D\in L^{2+\alpha}(f^\star d\mu)$ for some $\alpha>0$.  Moreover,
$
        \int_0^1 \sqrt{\log\mathcal{N}(\mathcal{D}_q,u)}\,du<\infty.
$
\item $\mathcal{\bar D}_q^c\backslash\mathcal{\bar D}_{q^\star}$
is nonempty.
\end{enumerate}
Let $\eta(q)>0$ be an arbitrary positive function.  Then
$$
	\limsup_{n\to\infty}\frac{1}{\log\log n}
	\sup_{q\ge q^\star}\frac{1}{\eta(q)}
        \left\{
        \sup_{f\in\mathcal{M}_q}\ell_n(f) -
        \sup_{f\in\mathcal{M}_{q^\star}}\ell_n(f)
        \right\} \ge C
$$
$\mathbf{P}^\star$-a.s., where $C>0$ is nonrandom but may depend on
$f^\star,\eta$.
\end{cor}

The proofs of Theorem \ref{thm:lowerlil} and Corollary
\ref{cor:lowerlil} are given in Appendix \ref{sec:prooflowerlil} below.

The fact that the geometric assumptions in Theorem \ref{thm:lowerlil} 
and Corollary \ref{cor:lowerlil} are expressed in terms of weighted 
classes is not surprising, as the sharp asymptotic expression provided 
by Theorem \ref{thm:lowerlil} for the pathwise fluctuations of the 
likelihood are expressed in terms of a variational problem on the 
weighted classes.  Nonetheless, we are naturally led to ask whether 
there is any relation between the geometric assumptions imposed in the 
upper bound Theorem \ref{thm:upperlil} and the lower bound Theorem 
\ref{thm:lowerlil}, which appear to be quite different at first sight. 
In \cite{GasHanGeo}, we show that the global entropy of the weighted 
class is closely related to local entropy, so that the geometric 
assumptions for the upper and lower bounds are not too far apart.

\begin{rem}
\label{rem:dimensionfree}
When $\mathcal{\bar D}_q$ and $\mathcal{\bar D}_p$ each contain an 
$L^2(f^\star d\mu)$-dense subset of continuously accessible points (which 
is typically the case in sufficiently smooth models), then Theorem 
\ref{thm:lowerlil} provides the exact characterization
\begin{multline*}
        \limsup_{n\to\infty}\frac{1}{\log\log n}\left\{
        \sup_{f\in\mathcal{M}_q}\ell_n(f)-
        \sup_{f\in\mathcal{M}_p}\ell_n(f)\right\} = \mbox{}\\
        \sup_{g\in L^2_0(f^\star d\mu)}\left\{
        \sup_{f\in\mathcal{\bar D}_q}(\langle f,g\rangle)_+^2
        -\sup_{f\in\mathcal{\bar D}_p}(\langle f,g\rangle)_+^2
        \right\}\quad
        \mathbf{P}^\star\mbox{\rm-a.s.}
\end{multline*}
Beside its intrinsic interest, this result has a surprising consequence.  
In the case that $\mathcal{M}_q$ and $\mathcal{M}_p$ are regular 
parametric models with $\dim(\mathcal{M}_q)> \dim(\mathcal{M}_p)$, one 
can choose $g\in\mathcal{\bar D}_q$ which is orthogonal to 
$\mathcal{\bar D}_p$.  As $\mathcal{\bar D}_q,\mathcal{\bar 
D}_p\subseteq L^2_0(f^\star d\mu)$ (see the proof of Corollary 
\ref{cor:lowerlil}), it follows easily that in this case the right-hand 
side of the previous equation display is precisely equal to $1$.  In 
particular, we obtain the curious conclusion that in regular parametric 
models, the magnitude of the fluctuations of the likelihood ratio 
statistic does not depend on the dimensions $\dim(\mathcal{M}_q)$ and 
$\dim(\mathcal{M}_p)$.  In contrast, it is well known that in regular 
parametric models, the likelihood ratio statistic itself converges 
weakly to a chi-square distribution with 
$\dim(\mathcal{M}_q)-\dim(\mathcal{M}_p)$ degrees of freedom, so the 
tails of the distribution of the likelihood ratio statistic do in fact 
depend strongly on the dimensions $\dim(\mathcal{M}_q)$ and 
$\dim(\mathcal{M}_p)$.  Of course, the dimension independence of the
pathwise fluctuations will also cease to hold if we are interested in
a result that is uniform in the order $q$, as in Theorem \ref{thm:upperlil}.
This highlights the fact that the problems investigated in this paper
are fundamentally different depending on whether or not one assumes a
prior upper bound on the model order.
\end{rem}

\section{Strongly consistent order estimation}
\label{sec:consist}

The goal of this section is to apply the results of section 
\ref{sec:pathwise} to identify what penalties and cutoffs yield strongly 
consistent order estimators.  We first develop some general consistency 
and inconsistency results, and then consider specifically the 
challenging problem of mixture order estimation.

\subsection{Consistency and minimal penalties}

In this section we consider the general setting introduced in section 
\ref{sec:setting}.  We now suppose, however, that the true model order 
$q^\star$ (as well as the true density $f^\star$) is not known, so that 
we must estimate $q^\star$ from an observation sequence 
$(X_k)_{k\ge 1}$.
To this end, define the \emph{penalized likelihood order estimator}
$$
        \hat q_n = \argmax_{q\ge 1}\left\{
        \sup_{f\in\mathcal{M}_q^n}\ell_n(f) - \pen(n,q)\right\},
$$
where $\pen(n,q)$ is a penalty function.  Our goal is to show that the 
penalized likelihood order estimator is strongly consistent, that is, 
$\hat q_n\to q^\star$ as $n\to\infty$ $\mathbf{P}^\star$-a.s., for a 
suitable choice of the penalty (that does not depend on $q^\star$ or 
$f^\star$).  Let us emphasize that the maximum in the definition of 
$\hat q_n$ is taken over \emph{all} model orders $q\ge 1$, that 
is, we do not assume that an a priori upper bound on the order is 
available, in contrast to most previous work on this topic.

We obtain the following general result.

\begin{thm}
\label{thm:mainconsist}
Suppose that for all $n$ sufficiently large
$$
        \mathcal{N}(\mathcal{H}_q^n(\varepsilon),\delta)
        \le
        \left(\frac{K(n)\varepsilon}{\delta}\right)^{\eta(q)}
$$
for all $q\ge q^\star$ and $\delta\le\varepsilon$,
where $K(n)\ge 1$ and $\eta(q)\ge q$ are increasing functions
and we assume that $\log K(n)=o(n)$.  Let $\pen(n,q)$ be a penalty 
that is increasing in $q$ and
\begin{align*}
   &     \lim_{n\to\infty}\sup_{q>q^\star}
        \frac{\eta(q)\{\log K(2n)\vee \log\log n\}}{\pen(n,q)-
        \pen(n,q^\star)}=0, \\
   &     \lim_{n\to\infty}\max_{q<q^\star}\frac{\pen(n,q)}{n}=0.
\end{align*}
Then $\hat q_n\to q^\star$ as $n\to\infty$ $\mathbf{P}^\star$-a.s.
\end{thm}

Theorem \ref{thm:mainconsist} is proved in Appendix \ref{sec:proofmainconsist}.

Let us now specialize to the case that $\mathcal{M}_q^n=\mathcal{M}_q$
does not depend on $n$, as in section \ref{sec:lowerlil}.  In this case,
Theorem \ref{thm:mainconsist} immediately yields the following corollary.

\begin{cor}
\label{cor:mainconsist}
Suppose that 
for all $q\ge q^\star$ and $\delta\le\varepsilon$
$$
        \mathcal{N}(\mathcal{H}_q(\varepsilon),\delta)
        \le
        \left(\frac{K\varepsilon}{\delta}\right)^{\eta(q)},
$$
where $K\ge 1$ and $\eta(q)\ge q$ is a strictly increasing function.
Define the penalty
$$
	\pen(n,q) = \eta(q)\,\varpi(n),
$$
where $\varpi(n)$ is any function such that
$$
	\lim_{n\to\infty}\frac{\log\log n}{\varpi(n)}=0,
	\qquad\quad
	\lim_{n\to\infty}\frac{\varpi(n)}{n}=0.
$$
Then $\hat q_n\to q^\star$ as $n\to\infty$ $\mathbf{P}^\star$-a.s.
\end{cor}

Corollary \ref{cor:mainconsist} states that, when $\mathcal{M}_q^n= 
\mathcal{M}_q$ does not depend on $n$, the penalized likelihood order 
estimator is strongly consistent provided the penalty grows faster than 
$\log\log n$ and slower than $n$.  Clearly the $\log\log n$ rate is the 
minimal one attainable by applying Theorem \ref{thm:mainconsist}. This 
raises the question whether the $\log\log n$ rate is indeed minimal, in 
the sense that smaller penalties yield inconsistent estimators. The 
following result shows that this is indeed the case, so that the 
result of Corollary \ref{cor:mainconsist} is essentially optimal.

\begin{cor}
\label{cor:inconsist}
Suppose there exists $q>q^\star$ such that 
\begin{enumerate}
\item
there is an envelope function $D:E\to\mathbb{R}$ such that
$|d|\le D$ for all $d\in\mathcal{D}_q$,
$D\in L^{2+\alpha}(f^\star d\mu)$ for some $\alpha>0$, and
$
        \int_0^1 \sqrt{\log\mathcal{N}(\mathcal{D}_q,u)}\,du<\infty
$;
\item $\mathcal{\bar D}_q^c\backslash\mathcal{\bar D}_{q^\star}$
is nonempty.
\end{enumerate}
Let $\eta(q)>0$ be any strictly increasing function, and let
$$
	\pen(n,q) = C\,\eta(q)\,\log\log n.
$$
If $C>0$ is sufficiently small,
$\hat q_n\ne q^\star$ infinitely often $\mathbf{P}^\star$-a.s.
\end{cor}

The proof of Corollary \ref{cor:inconsist} is given in Appendix
\ref{sec:proofmainconsist}.  Let us note that the proof of Corollary 
\ref{cor:inconsist} actually shows that 
$\sup_{f\in\mathcal{M}_q}\ell_n(f)-\pen(n,q)>\sup_{f\in\mathcal{M}_{q^\star}} 
\ell_n(f)-\pen(n,q^\star)$ infinitely often $\mathbf{P}^\star$-a.s., so 
Corollary \ref{cor:inconsist} is not altered even if 
we were to impose a prior upper bound on the order.

In conclusion, we have shown that when $\mathcal{M}_q^n=\mathcal{M}_q$ 
does not depend on $n$, penalties growing faster than $\log\log n$ are 
consistent while the penalty $C\,\eta(q)\log\log n$ is inconsistent when 
the constant $C$ is sufficiently small.  From the proof of Theorem 
\ref{thm:mainconsist}, we can also see that the penalty 
$C\,\eta(q)\log\log n$ is consistent when $C$ is sufficiently large. 
However, the critical value of $C$ may depend on the unknown parameter 
$f^\star$, so that this \emph{minimal} penalty may not be implementable.  
On the other hand, assuming that $\eta(q)$ does not depend on $f^\star$ 
(as is typically the case), penalties satisfying the assumptions of 
Theorem \ref{thm:mainconsist} obviously do not depend on the unknown 
parameter $f^\star$ and therefore define admissible estimators.  When
 $\mathcal{M}_q^n$ depends on $n$, larger penalties may be 
required to ensure consistency, depending on the growth rate of $K(n)$.

\subsection{Location mixture order estimation}
\label{sec:mixorder}

We finally apply our general results to location mixture order 
estimation.  Throughout this section, let $E=\mathbb{R}^d$ and let $\mu$ 
be the Lebesgue measure on $\mathbb{R}^d$.  Fix a strictly positive 
probability density $f_0$ with respect to $\mu$, and define
$$
	\mathcal{M}_q^n = \left\{
	\sum_{i=1}^q\pi_if_{\theta_i}:\pi_i\ge 0,~\sum_{i=1}^q\pi_i=1,~
	\theta_i\in\Theta(n)\right\},
$$
where $f_\theta(x)=f_0(x-\theta)$ and
$\cdots\subseteq\Theta(n)\subseteq\Theta(n+1)\subseteq\cdots
\subset\mathbb{R}^d$ is an increasing family of bounded subsets of
$\mathbb{R}^d$.  We fix $f^\star\in\mathcal{M}$ throughout this section.
Let
\begin{align*}
        H_{0}(x)&=\sup_{\theta\in\Theta}f_{\theta}(x)/f^{\star}(x),\\
        H_{1}(x)&=\sup_{\theta\in\Theta}
        \max_{i=1,\ldots,d}| 
        \partial f_\theta(x) / \partial\theta^i|/f^{\star}(x),\\        H_{2}(x)&=\sup_{\theta\in\Theta}
        \max_{i,j=1,\ldots,d}| 
        \partial^{2} f_\theta(x) / \partial\theta^i\partial\theta^j|/f^{\star}(x),
        \\        H_{3}(x)&=
        \sup_{\theta\in\Theta}\max_{i,j,k=1,\ldots,d}| 
        \partial^{3} f_\theta(x) / \partial\theta^i\partial\theta^j\partial\theta^k|/f^{\star}(x)
\end{align*}
when $f_0$ is sufficiently differentiable, and let
\begin{aspta}
The following hold:
\begin{enumerate}
\item $f_0\in C^3$ and $f_0(x)$, $(\partial f_0/\partial\theta^i)(x)$ vanish as 
$\|x\|\to\infty$.
\item $H_{k}\in L^{4}(f^{\star}d\mu)$ for $k=0,1,2$ and $H_{3}\in L^{2}(f^{\star}d\mu)$.
\end{enumerate} 
\end{aspta}

In the following, we consider two separate cases.  The first case is 
that of a compact parameter set, where $\Theta(n)=\Theta$ does not 
depend on $n$. In this setting, we obtain a general result.  Then, we 
consider the noncompact case in the setting of Gaussian mixtures,
and illustrate how Theorem \ref{thm:mainconsist} can be used to obtain
consistency results in this case.
To be able to use Theorem \ref{thm:mainconsist}, we need suitable estimates
on the local entropy of mixtures. The following result is 
given in \cite{GasHanGeo}.

\begin{thm}
\label{cor:localmixtures}
Suppose that Assumption A holds.  Then if $\Theta(n)=\Theta$ is a bounded subset of $\mathbb{R}^d$ with diameter $2T$,
$$
        \mathcal{N}(\mathcal{H}_{q}^{n}(\varepsilon),\delta) \le
        \left(
        \frac{C_\Theta\,\varepsilon}{\delta}
        \right)^{18(d+1)q+1} 
$$
for all $q\ge q^\star$ and $\delta/\varepsilon\le 1$, where
$$
        C_\Theta = L^\star\, (T\vee 1)^{1/6} \,
        (\|H_0\|_4^4\vee\|H_1\|_4^4\vee\|H_2\|_4^4\vee\|H_3\|_2^2)^{5/4}
$$
and $L^\star$ is a constant that depends only on $d$, $q^\star$ and $f^\star$.
\end{thm}

\begin{example}[Gaussian mixtures]
\label{ex:gaussian}
Consider mixtures of standard Gaussian densities
$f_0(x)=(2\pi)^{-d/2}e^{-\|x\|^2/2}$, and let
$\Theta(T) = \{\theta\in\mathbb{R}^d:\|\theta\|\le T\}$.  Fix a nondegenerate
mixture $f^\star$, and define $T^\star=
\max_{i=1,\ldots,q^\star}\|\theta_i^\star\|$.  Denote by
$\mathcal{H}_q(\varepsilon,T)$ the Hellinger ball associated to the
parameter set $\Theta(T)$.  Then 
$$
        \mathcal{N}(\mathcal{H}_{q}(\varepsilon,T),\delta) \le
        \left(
        \frac{C_1^\star e^{C_2^\star T^2}\varepsilon}{\delta}
        \right)^{18(d+1)q+1} 
$$
for all $q\ge q^\star$, $T\ge T^\star$, and $\delta/\varepsilon\le 1$,
where $C_1^\star,C_2^\star$ are constants that depend on $d$, $q^\star$ and
$f^\star$ only.  To prove this, it suffices to show that
Assumption A holds and that $\|H_k\|_4$ for $k=0,1,2$ and $\|H_3\|_2$
are of order $e^{CT^2}$.  These facts are readily verified by a
straightforward computation.
\end{example}

Let us first consider the case of a compact parameter set.  We 
obtain a general consistency result under Assumption A.

\begin{prop}
\label{prop:mixconsist}
Suppose that the parameter set
$\Theta(n)=\Theta$ is a bounded subset of $\mathbb{R}^d$
independent of $n$, and that Assumption A holds.
If we choose a penalty of the form
$$
	\pen(n,q) = q\,\omega(n),\qquad
	\lim_{n\to\infty}\frac{\log\log n}{\omega(n)}=
	\lim_{n\to\infty}\frac{\omega(n)}{n}=0,
$$
then $\hat q_n\to q^\star$ as $n\to\infty$ $\mathbf{P}^\star$-a.s.
On the other hand, if 
$$
	\pen(n,q)=C\,q\,\log\log n
$$
where $C>0$ is a sufficiently small constant, then we have
$\hat q_n\ne q^\star$ infinitely often $\mathbf{P}^\star$-a.s.
\end{prop}

We therefore find that in the setting of location mixtures with a 
compact parameter set, the minimal penalty is of order $\log\log n$.
Moreover, the popular BIC penalty
\begin{equation}
\label{eq:bic}
	\pen(n,q) = \frac{dq+q-1}{2}\,\log n
\end{equation}
yields a strongly consistent mixture order estimator in this setting, 
without a prior upper bound on the order.  The requisite Assumption A 
is mild, which highlights the broad applicability of this 
result.  However, the assumption of a compact parameter space can be 
quite restrictive in practice.

Let us therefore consider a case where the parameter space is noncompact.
For simplicity we restrict our attention to Gaussian mixtures, that is,
we choose $f_0(x) = (2\pi)^{-d/2}e^{-\|x\|^2/2}$, and we choose the 
restricted parameter sets $\Theta(n) = \{\theta\in\mathbb{R}^d:\|\theta\|\le
T(n)\}$ for some sequence $T(n)\uparrow\infty$.  Our aim is to choose the
penalty $\pen(n,q)$ and cutoff $T(n)$ so that the penalized likelihood
order estimator is strongly consistent. 

 In this setting, we obtain the
following result.

\begin{prop}
\label{prop:mixconsist2}
For the case $f_0(x) = (2\pi)^{-d/2}e^{-\|x\|^2/2}$ and 
$\Theta(n) = \{\theta\in\mathbb{R}^d:\|\theta\|\le T(n)\}$,
consider a penalty of the form $\pen(n,q)=q\,\omega(n)$.
If
$$
	\lim_{n\to\infty}\frac{\log\log n}{\omega(n)}=
	\lim_{n\to\infty}\frac{\omega(n)}{n}=0,\quad
	T(n) = O(\sqrt{\log\log n}),
$$
then $\hat q_n\to q^\star$ as $n\to\infty$ $\mathbf{P}^\star$-a.s.

On the other hand, the BIC penalty (\ref{eq:bic}) yields a strongly
consistent order estimator if $T(n) = o(\sqrt{\log n})$.
\end{prop}

This result illustrates that our theory can establish consistency of the 
penalized likelihood mixture order estimator without any prior upper 
bounds on the model order or the magnitude of the true parameters. Let 
us note that there is nothing particularly special about the Gaussian 
case: a similar result can be obtained, in principle, for any mixture 
distribution $f_0$, as long as one can obtain suitable estimates on the 
quantities $\|H_i\|_4$ that appear in Theorem \ref{cor:localmixtures} 
(see Example \ref{ex:gaussian} for the Gaussian case).

The proofs of Propositions \ref{prop:mixconsist} and 
\ref{prop:mixconsist2} are given in Appendix \ref{sec:proofmixconsist} 
below.

\appendix

\section{Proof of Theorem \ref{thm:upperlil}}
\label{sec:proofupperlil}

The proof of Theorem \ref{thm:upperlil} is based on the following 
deviation bound for the log-likelihood ratio.  This bound is essentially 
from \cite[Corollary 7.5]{vdG00}, but the additional maximum inside the 
probability is essential for our purposes.

\begin{thm}
\label{thm:deviationbound}
Let $\mathcal{M}$ be a family of strictly positive 
probability densities with respect to
a reference measure $\mu$, fix some $f^\star\in\mathcal{M}$, and define 
the Hellinger ball $\mathcal{H}(\varepsilon) = \{\sqrt{f/f^\star}:
f\in\mathcal{M},~h(f,f^\star)\le\varepsilon\}$ where 
$h(f,g)^2 = \int (\sqrt{f}-\sqrt{g})^2d\mu$.  Suppose that
for some constants $K\ge 1$, $p\ge 1$ and all
$\delta\le\varepsilon$
$$
        \mathcal{N}(\mathcal{H}(\varepsilon),\delta)
        \le
        \left(
        \frac{K\varepsilon}{\delta}
        \right)^p,
$$
where $\mathcal{N}(\mathcal{H}(\varepsilon),\delta)$ is the minimal number
of brackets of $L^2(f^\star d\mu)$-width $\delta$ needed to cover
$\mathcal{H}(\varepsilon)$.  Let $(X_i)_{i\in\mathbb{N}}$ be i.i.d.\
with distribution $f^\star d\mu$.  Then
$$
        \mathbf{P}\left[
        \max_{n\le k\le 2n}
        \sup_{f\in\mathcal{M}}
        \sum_{j=1}^k\log\left(\frac{f(X_j)}{f^\star(X_j)}\right)
        \ge \alpha
        \right] \le
        C\,e^{-\alpha/C}
$$
for all $\alpha\ge Cp(1+\log K)$, $n\ge 1$
\emph{[}$C$ is a universal constant\emph{]}.
\end{thm}

\begin{proof}
Define $\bar f = (f+f^\star)/2$ for any $f\in\mathcal{M}$, and define
the empirical process
$\nu_n(g) = n^{-1/2}\sum_{k=1}^n\{g(X_k)-\mathbf{E}[g(X_k)]\}$.
Using concavity of $\log x$ we have
$$
        \sum_{j=1}^k\log\left(\frac{f(X_j)}{f^\star(X_j)}\right)
        \le
        2k^{1/2} \nu_k(\log(\bar f/f^\star)) - 2k D(f^\star||\bar f),
$$
where $D(f^\star||f)= \int \log(f^\star/f) f^\star d\mu$ is relative
entropy.  As $D(f^\star||f) \ge h(f,f^\star)^2$, we can estimate
\begin{align*}
        &\mathbf{P}\Bigg[
        \max_{n\le k\le 2n}
        \sup_{f\in\mathcal{M}}
        \sum_{j=1}^k\log\left(\frac{f(X_j)}{f^\star(X_j)}\right)
        \ge \alpha
        \Bigg] \\
        &\le
        \mathbf{P}\Bigg[
        \max_{n\le k\le 2n}
        \sup_{f\in\mathcal{M}}
        \{
        \sqrt{k}\,\nu_k(\log(\bar f/f^\star)) - k h(\bar f,f^\star)^2
        \}
        \ge \frac{\alpha}{2}
        \Bigg] \\
        &\le
        \sum_{s=0}^S \mathbf{P}\Bigg[
        \max_{n\le k\le 2n}
        \sup_{f\in\mathcal{M}^s}
        |\sqrt{k}\,\nu_k(\log(\bar f/f^\star))|
        \ge \alpha 2^{s-1}
        \Bigg] \\
        &\le
        3\sum_{s=0}^S \max_{n\le k\le 2n} \mathbf{P}\Bigg[
        \sup_{f\in\mathcal{M}^s}
        |\nu_k(\log(\textstyle{\sqrt{\bar f/f^\star}}))|
        \ge  
        \alpha 2^{s-5}/\sqrt{n}
        \Bigg]
\end{align*}
where $\mathcal{M}^s=\{f\in\mathcal{M}:n h(\bar f,f^\star)^2\le\alpha 2^s\}$,
$S = \min\{s:\alpha 2^s n^{-1}>2\}$, and we have used 
Lemma \ref{lem:etemadi} below for the last inequality.  The remainder
of the proof is identical to that of \cite[Theorem 7.4]{vdG00}, provided
we show that for
$\mathcal{\bar H}(\varepsilon) = \{\sqrt{\bar f/f^\star}:
f\in\mathcal{M},~h(\bar f,f^\star)\le\varepsilon\}$
$$
        \mathcal{N}(\mathcal{\bar H}(\varepsilon),\delta)
        \le
        \left(
        \frac{2\sqrt{2}K\varepsilon}{\delta}
        \right)^p.
$$
To this end, fix $\delta\le\varepsilon$, and note that 
$h(f,f^\star)\le 4h(\bar f,f^\star)$ by \cite[Lemma 4.2]{vdG00},
so that $\{f\in\mathcal{M}:h(\bar f,f^\star)\le\varepsilon\}
\subseteq\{f\in\mathcal{M}:h(f,f^\star)\le 4\varepsilon\}$.
By assumption, there exist $N\le (2\sqrt{2}K\varepsilon/\delta)^p$ and 
functions $g_1,\ldots,g_N,
h_1,\ldots,h_N$ such that $\|h_i-g_i\|_2\le\delta\sqrt{2}$ for every $i$,
and for every $u\in \mathcal{H}(4\varepsilon)$ there is an $i$ such
that $g_i\le u\le h_i$.  But for every $f\in\mathcal{M}$ such that
$h(\bar f,f^\star)\le\varepsilon$, we then have for some $i$
$$
        2^{-1/2}\sqrt{g_i^2 + 1} \le
        \sqrt{\bar f/f^\star }
        \le 2^{-1/2}\sqrt{h_i^2+1}.
$$
Using $|\sqrt{a+c}-\sqrt{b+c}|\le |\sqrt{a}-\sqrt{b}|$
for $a,b,c\ge 0$ we have
$$
        \left\|2^{-1/2}\sqrt{h_i^2+1} - 2^{-1/2}\sqrt{g_i^2 + 1}\right\|_2
        \le 2^{-1/2} \|h_i-g_i\|_2 \le \delta.
$$
The result now follows directly.
\end{proof}

The following variant of Etemadi's inequality was used in the proof. The 
proof follows closely that of the classical Etemadi inequality, see 
\cite[Appendix M19]{Bil99}.

\begin{lem}
\label{lem:etemadi}
Let $\mathcal{Q}$ be a family of measurable 
functions $f:E\to\mathbb{R}$.  Then we have for
every $\alpha>0$ and $m,n\in\mathbb{N}$, $m\le n$
$$
        \mathbf{P}^\star\left[
        \max_{k=m,\ldots,n}
        \sup_{f\in\mathcal{Q}}|S_k(f)|
        \ge 3\alpha
        \right] 
	\le
        3\max_{k=m,\ldots,n}
        \mathbf{P}^\star\left[
        \sup_{f\in\mathcal{Q}}|S_k(f)|
        \ge \alpha
        \right],
$$
where $S_n(f)=n^{1/2}\nu_n(f)$.  
\end{lem}

\begin{proof}
Define the stopping time
$$
	\tau=\inf\left\{k\ge m:\sup_{f\in\mathcal{Q}}|S_k(f)|\ge 3\alpha
	\right\}.
$$
Then
\begin{align*}
        &\mathbf{P}^\star\left[
        \max_{k=m,\ldots,n}
        \sup_{f\in\mathcal{Q}}|S_k(f)|
        \ge 3\alpha
        \right] =
        \mathbf{P}^\star[\tau\le n] 
	\\ &
        \le
        \mathbf{P}^\star\left[
        \sup_{f\in\mathcal{Q}}|S_n(f)|\ge\alpha
        \right] + 
        \sum_{k=m}^{n}
        \mathbf{P}^\star\left[\tau=k\mbox{\rm\ and }
        \sup_{f\in\mathcal{Q}}|S_n(f)|<\alpha\right].
\end{align*}
But on the event $\{\tau=k\mbox{\rm\ and }
\sup_{f\in\mathcal{Q}}|S_n(f)|<\alpha\}$, we have
$$
        2\alpha \le \sup_{f\in\mathcal{Q}}|S_k(f)|
        -\sup_{f\in\mathcal{Q}}|S_n(f)| \le
        \sup_{f\in\mathcal{Q}}|S_k(f)-S_n(f)|.
$$
Therefore, we can estimate
\begin{align*}
        &\mathbf{P}^\star\left[
        \max_{k=m,\ldots,n}
        \sup_{f\in\mathcal{Q}}|S_k(f)|
        \ge 3\alpha
        \right]  \\ 
        &\mbox{}\le
        \mathbf{P}^\star\left[
        \sup_{f\in\mathcal{Q}}|S_n(f)|\ge\alpha
        \right] + 
        \sum_{k=m}^{n}
        \mathbf{P}^\star\left[\tau=k\mbox{\rm\ and }
        \sup_{f\in\mathcal{Q}}|S_n(f)-S_k(f)|\ge 2\alpha
        \right] \\
        &\le
        \mathbf{P}^\star\left[
        \sup_{f\in\mathcal{Q}}|S_n(f)|\ge\alpha
        \right] + 
        \max_{k=m,\ldots,n}
        \mathbf{P}^\star\left[\sup_{f\in\mathcal{Q}}|S_n(f)-S_k(f)|\ge 2\alpha
        \right],
\end{align*}
where we have used that $\sup_{f\in\mathcal{Q}}|S_n(f)-S_k(f)|$ 
and $\{\tau=k\}$ are independent to obtain the last inequality.
The remainder of the proof is now easily completed.
\end{proof}

We can now complete the proof of Theorem \ref{thm:upperlil}.

\begin{proof}[Proof of Theorem \ref{thm:upperlil}]
By assumption, we have $f^\star\in\mathcal{M}_q^n$ for all $q\ge q^\star$
when $n$ is sufficiently large.  Then by Theorem \ref{thm:deviationbound},
we have for $n$ sufficiently large
$$
        \mathbf{P}^\star\left[
        \max_{n\le k\le 2n} \sup_{f\in\mathcal{M}_q^{2n}}
        \{\ell_k(f)-\ell_k(f^\star)\}\ge\alpha
        \right]\le
        C\,e^{-\alpha/C}
$$
for all $\alpha\ge C\eta(q)(1+\log K(2n))$ and $q\ge q^\star$.
Define
$$
	\Delta_k(q,q^\star) = 
        \sup_{f\in\mathcal{M}_q^k}
        \ell_k(f)-\sup_{f\in \mathcal{M}_{q^\star}^k}\ell_k(f).
$$
Using that $\mathcal{M}_q^k\subseteq\mathcal{M}_q^{2n}$ for
$n\le k\le 2n$ and $\ell_k(f^\star)\le\sup_{f\in \mathcal{M}_{q^\star}^k}
\ell_k(f)$, we have for $n$ sufficiently large
$$
        \mathbf{P}^\star\Bigg[
        \max_{n\le k\le 2n}
        \sup_{q\ge q^\star}
        \frac{1}{\eta(q)}
	\Delta_k(q,q^\star)
	\ge\alpha
        \Bigg] \le
        \sum_{q=q^\star}^\infty
        C\,e^{-\alpha\eta(q)/C}
$$
for all $\alpha\ge C(1+\log K(2n))$.  Let $\beta(n)$ be an increasing
function.  Then for all $n$ sufficiently large
$$
        \mathbf{P}^\star\left[
        \max_{2^n\le k\le 2^{n+1}}
        \frac{1}{\beta(k)}
        \sup_{q\ge q^\star}
        \frac{1}{\eta(q)}
	\Delta_k(q,q^\star)  
	\ge 2C
        \right]\le
        \frac{2C}{n^2},
$$
provided that
$\beta(2^n)\ge\log K(2^{n+1}) \vee \log\log 2^n$.  The proof is now
easily completed using the Borel-Cantelli lemma.
\end{proof}

\section{Proof of Theorem \ref{thm:lowerlil}}
\label{sec:prooflowerlil}

The proof of Theorem \ref{thm:lowerlil} is based on a sequence of 
auxiliary results.  First, we will need a compact law 
of iterated logarithm for the Strassen functional
$$
        I_n(g) = \frac{1}{\sqrt{2n\log\log n}}\sum_{i=1}^n
        \left\{g(X_i)-\mathbf{E}^\star(g(X_1))\right\}.
$$
We state the requisite result for future reference.

\begin{thm}
\label{thm:clil}
Let $\mathcal{Q}$ be a family of measurable functions from $E$
to $\mathbb{R}$ such that
$$
        \int_0^1\sqrt{\log\mathcal{N}(\mathcal{Q},u)}\,du<\infty.
$$
Then, $\mathbf{P}^\star$-a.s., the sequence $(I_n)_{n\ge 0}$ is relatively 
compact in $\ell_\infty(\mathcal{Q})$, and its set of cluster points 
coincides precisely with the set 
$\mathcal{K} = \{f\mapsto \langle f,g\rangle:g\in L^2_0(f^\star d\mu)\}$.
\end{thm}

Proofs of this result can be found in \cite[Theorem 4.2]{Oss87} or in 
\cite[Theorem 9]{LT89}.  We will also need the following simple
well-known fact, whose proof is omitted.

\begin{lem}
\label{lem:maximal}
Let $(X_i)_{i\ge 1}$ be an i.i.d.\ sequence such that
$\mathbf{E}[|X_1|^p]<\infty$.  Then 
$n^{-1/p}\max_{i=1,\ldots,n}|X_i|\to 0$ a.s.
\end{lem}

Finally, we will need the following likelihood inequality that relates 
the log-likelihood ratio $\ell_n(f)-\ell_n(f^\star)$ to the empirical 
process.  Related inequalities appear in \cite{Gas02,LiuShao03,Cha06}, 
but the following form is perhaps the most natural.

\begin{lem}
\label{lem:likineq}
For any probability density $f\ne f^\star$
$$
        \ell_n(f)-\ell_n(f^\star) \le 
        |\nu_n(d_f)|^2,
$$
where $\nu_n(g) = n^{-1/2}\sum_{k=1}^n\{g(X_k)-\mathbf{E}^\star[g(X_k)]\}$.
\end{lem}

\begin{proof}
Note that
$$
        h(f,f^\star)^2 = 2 - \int 2\sqrt{ff^\star}\,d\mu
        = -2\,h(f,f^\star)\,\mathbf{E}^\star(d_f(X_1)).
$$
Using $\log(1+x)\le x$, we can estimate
\begin{align*}
        \ell_n(f)-\ell_n(f^\star) 
        &=   \sum_{i=1}^n 2\,\log(1+h(f,f^\star)\,d_f(X_i)) \\
        &\le \sum_{i=1}^n 2\,h(f,f^\star)\,d_f(X_i) \\
        & =
        2\,\nu_n(d_f)\,h(f,f^\star)\,\sqrt{n} - h(f,f^\star)^2\,n \\
        &\le \sup_{p\in\mathbb{R}}\left\{2\,\nu_n(d_f)\,p - p^2
        \right\}.
\end{align*}
The proof is easily completed.
\end{proof}

We can now obtain the following asymptotic expansion of the 
log-likelihood, which provides a pathwise counterpart to the
weak convergence theory in \cite{Gas02,LiuShao03}.

\begin{prop}
\label{prop:asympexp}
Let $q\ge q^\star$.  Assume that
$$
        \int_0^1 \sqrt{\log\mathcal{N}(\mathcal{D}_q,u)}\,du<\infty.
$$
Moreover, suppose that $|d|\le D$ for all $d\in\mathcal{D}_q$ with
$D\in L^{2+\alpha}(f^\star d\mu)$ for some $\alpha>0$.
Then
\begin{multline*}
        \sup_{f\in\mathcal{M}_q(4\sqrt{\log\log n/n})}
        \bigg\{
        2\,I_n(d_f)\,h(f,f^\star)\,\sqrt{\frac{2n}{\log\log n}} 
        - h(f,f^\star)^2\,\frac{2n}{\log\log n}\bigg\} \\
        -\frac{1}{\log\log n}\left\{
        \sup_{f\in\mathcal{M}_q}\ell_n(f)-
        \ell_n(f^\star)\right\}
        \xrightarrow{n\to\infty}0\quad\mathbf{P}^\star\mbox{\rm-a.s.},
\end{multline*}
where we have defined $\mathcal{M}_q(\varepsilon)=
\{f\in\mathcal{M}_q:h(f,f^\star)\le\varepsilon\}$.
\end{prop}

\begin{proof}
We proceed in several steps.

\textbf{Step 1} (\emph{localization}).
As $q\ge q^\star$ (hence $f^\star\in\mathcal{M}_q$), clearly
$$
        \sup_{f\in\mathcal{M}_q}\ell_n(f)-\ell_n(f^\star)
        = 
        \sup_{f\in\mathcal{M}_q:\ell_n(f)-\ell_n(f^\star)\ge 0}
        \left\{\ell_n(f)-\ell_n(f^\star)\right\}.
$$
Now note that, as in the proof of Lemma \ref{lem:likineq},
$$
        \ell_n(f)-\ell_n(f^\star) \le
        2\,\nu_n(d_f)\,h(f,f^\star)\,\sqrt{n} - h(f,f^\star)^2\,n.
$$
Therefore, we can estimate
\begin{align*}
        &
        \sup_{f\in\mathcal{M}_q:\ell_n(f)-\ell_n(f^\star)\ge 0}
        h(f,f^\star) 
        \\ &\qquad\mbox{}\le
        \sup_{f\in\mathcal{M}_q:\ell_n(f)-\ell_n(f^\star)\ge 0}
        \left\{
        h(f,f^\star) +
        \frac{\ell_n(f)-\ell_n(f^\star)}{n\,h(f,f^\star)}\right\}
        \\ &\qquad\mbox{}\le
        \frac{2}{\sqrt{n}}
        \sup_{f\in\mathcal{M}_q:\ell_n(f)-\ell_n(f^\star)\ge 0}
        \nu_n(d_f)
        \\ &\qquad\mbox{}\le \sqrt{\frac{8\log\log n}{n}}
        \sup_{d\in\mathcal{D}_q}I_n(d).
\end{align*}
Now note that we can estimate
$$
        \sup_{d\in\mathcal{D}_q}I_n(d) \le 
        \inf_{g\in L^2_0(f^\star d\mu)}
        \sup_{d\in\mathcal{D}_q}|I_n(d) - \langle d,g\rangle|
        + \sup_{d\in\mathcal{D}_q}\sup_{g\in L^2_0(f^\star d\mu)}
        \langle d,g\rangle.
$$
The first term on the right converges to zero $\mathbf{P}^\star$-a.s.\ as 
$n\to\infty$ by Theorem \ref{thm:clil}, while the second term is easily
seen to equal $\sup_{d\in\mathcal{D}_q}\|d-\langle 1,d\rangle\|_2\le 1$.  
Therefore
$$
        \sup_{f\in\mathcal{M}_q:\ell_n(f)-\ell_n(f^\star)\ge 0}
        h(f,f^\star) 
        \le (1+\varepsilon)\sqrt{\frac{8\log\log n}{n}}
$$
eventually as $n\to\infty$ $\mathbf{P}^\star$-a.s.\
for any $\varepsilon>0$.  In particular,
$$
        \{f\in\mathcal{M}_q:\ell_n(f)-\ell_n(f^\star)\ge 0\}
        \subseteq
        \left\{f\in\mathcal{M}_q:h(f,f^\star)\le 
        4\sqrt{\log\log n/n}\right\}
$$
eventually as $n\to\infty$ $\mathbf{P}^\star$-a.s.
This implies that 
$$
        \sup_{f\in\mathcal{M}_q}\ell_n(f)-\ell_n(f^\star) 
        \le
        \sup_{f\in\mathcal{M}_q:h(f,f^\star)\le 4\sqrt{\log\log n/n}}
        \left\{\ell_n(f)-\ell_n(f^\star)\right\}
$$
eventually as $n\to\infty$ $\mathbf{P}^\star$-a.s.
But the reverse inequality clearly holds for all $n\ge 0$,
so that in fact
$$
        \sup_{f\in\mathcal{M}_q}\ell_n(f)-\ell_n(f^\star)
        =
        \sup_{f\in\mathcal{M}_q(4\sqrt{\log\log n/n})}
        \left\{\ell_n(f)-\ell_n(f^\star)\right\}
$$
eventually as $n\to\infty$ $\mathbf{P}^\star$-a.s.

\textbf{Step 2} (\emph{Taylor expansion}).
Taylor expansion gives $2\log(1+x)=2x-x^2+x^2R(x)$, where $R(x)\to 0$ as
$x\to 0$.  Thus we can write, for any $f\in\mathcal{M}_q$,
\begin{multline*}
        \ell_n(f)-\ell_n(f^\star)
        = \sum_{i=1}^n 2\,\log(1+h(f,f^\star)\,d_f(X_i)) =
        \mbox{} \\
        2\,h(f,f^\star)\sum_{i=1}^n \left\{d_f(X_i)
        +\frac{1}{2}\,h(f,f^\star)\right\}
        -h(f,f^\star)^2\sum_{i=1}^n (d_f(X_i))^2 
        -n\,h(f,f^\star)^2 \\
        \mbox{}
        +h(f,f^\star)^2\sum_{i=1}^n (d_f(X_i))^2
        R(h(f,f^\star)\,d_f(X_i)).
\end{multline*}
Using that $\mathbf{E}^\star(d_f(X_1))=-h(f,f^\star)/2$, we therefore have
\begin{multline*}
        \frac{1}{\log\log n}\left\{
        \ell_n(f)-\ell_n(f^\star)\right\} = \mbox{}\\
        R_{f,n}\,\frac{n\,h(f,f^\star)^2}{\log\log n}
        + 2\,I_n(d_f)\,h(f,f^\star)\,\sqrt{\frac{2n}{\log\log n}}
        - h(f,f^\star)^2\,\frac{2n}{\log\log n}
\end{multline*}
where we have defined
$$
        R_{f,n} = \frac{1}{n}\sum_{i=1}^n \{1-(d_f(X_i))^2\} 
        +\frac{1}{n}\sum_{i=1}^n (d_f(X_i))^2
        R(h(f,f^\star)\,d_f(X_i)).
$$
It follows easily that
\begin{align*}
        &\Bigg|\sup_{f\in\mathcal{M}_q(4\sqrt{\log\log n/n})}
        \Bigg\{
        2\,I_n(d_f)\,h(f,f^\star)\,\sqrt{\frac{2n}{\log\log n}}
        - h(f,f^\star)^2\,\frac{2n}{\log\log n}\Bigg\} 
\\ &\qquad\qquad\mbox{}
        -\frac{1}{\log\log n}\left\{
        \sup_{f\in\mathcal{M}_q}\ell_n(f)-
        \ell_n(f^\star)\right\}\Bigg|
\\ &\quad\mbox{}
        \le
        \sup_{f\in\mathcal{M}_q(4\sqrt{\log\log n/n})}
        |R_{f,n}|\,\frac{n\,h(f,f^\star)^2}{\log\log n}
\\ &\quad\mbox{}
        \le
        16\sup_{f\in\mathcal{M}_q(4\sqrt{\log\log n/n})}|R_{f,n}|
\end{align*}
eventually as $n\to\infty$ $\mathbf{P}^\star$-a.s.

\textbf{Step 3} (\emph{end of proof}).  We can easily estimate
\begin{multline*}
        \sup_{f\in\mathcal{M}_q(4\sqrt{\log\log n/n})}|R_{f,n}| \le
        \sup_{f\in\mathcal{M}_q}
        \left|\frac{1}{n}\sum_{i=1}^n \{d_f(X_i)^2-1\}\right|
        \\ \mbox{}
        +
        \Bigg(
        \sup_{|x|\le 4\sqrt{\log\log n/n}\max_{i\le n}D(X_i)}
        |R(x)|\Bigg)\,
        \frac{1}{n}\sum_{i=1}^n (D(X_i))^2.
\end{multline*}
As $\mathcal{N}(\mathcal{D}_q,\delta)<\infty$ for every $\delta>0$,
the class $\{d^2:d\in\mathcal{D}_q\}$ can be covered by a finite number of 
brackets with arbitrary small $L^1(f^\star d\mu)$-norm and is therefore 
$\mathbf{P}^\star$-Glivenko-Cantelli.  Moreover, by construction 
$\mathbf{E}^\star[(d_f(X_i))^2]=1$ for all $f\in\mathcal{M}_q$.  
Therefore, the first term in this expression converges to zero as 
$n\to\infty$ $\mathbf{P}^\star$-a.s.  On the other hand, by Lemma 
\ref{lem:maximal} and the fact that $D\in L^{2+\alpha}(f^\star d\mu)$, 
we have $\mathbf{P}^\star$-a.s.
$$
        \sqrt{\log\log n/n}\,\max_{i=1,\ldots,n}D(X_i) = 
        \frac{\sqrt{\log\log n}}{n^{\alpha/2(2+\alpha)}}\,
        n^{-1/(2+\alpha)}\max_{i=1,\ldots,n}D(X_i) 
        \xrightarrow{n\to\infty}0.
$$
Therefore the second term converges to zero also, and the proof is 
evidently complete.
\end{proof}

\begin{prop}
\label{prop:asympexp1}
Let $q\ge q^\star$.  Assume that
$$
        \int_0^1 \sqrt{\log\mathcal{N}(\mathcal{D}_q,u)}\,du<\infty.
$$
Moreover, suppose that $|d|\le D$ for all $d\in\mathcal{D}_q$ with
$D\in L^{2+\alpha}(f^\star d\mu)$ for some $\alpha>0$.
Then
$$
        \liminf_{n\to\infty}\left\{
        \sup_{d\in\mathcal{\bar D}_q}(I_n(d))_+^2
        -\frac{1}{\log\log n}\left\{
        \sup_{f\in\mathcal{M}_q}\ell_n(f)-
        \ell_n(f^\star)\right\}\right\}
$$
is nonnegative $\mathbf{P}^\star$-a.s.
\end{prop}

\begin{proof}
By Proposition \ref{prop:asympexp}, we have
\begin{align*}
        &
        \liminf_{n\to\infty}\Bigg\{
        \sup_{d\in\mathcal{\bar D}_q}(I_n(d))_+^2
        -\frac{1}{\log\log n}\Bigg\{
        \sup_{f\in\mathcal{M}_q}\ell_n(f)-
        \ell_n(f^\star)\Bigg\}\!\Bigg\} \\
        &\ge 
        \liminf_{n\to\infty}\Bigg\{
        \sup_{d\in\mathcal{\bar D}_q}(I_n(d))_+^2
        -
        \sup_{f\in\mathcal{M}_q(4\sqrt{\log\log n/n})}
        \sup_{p\ge 0}
        \left\{
        2\,I_n(d_f)\,p
        - p^2\right\}
        \Bigg\} \\
        &=
        \liminf_{n\to\infty}\Bigg\{
        \sup_{d\in\mathcal{\bar D}_q}(I_n(d))_+^2
        - \sup_{f\in\mathcal{M}_q(4\sqrt{\log\log n/n})}
        (I_n(d_f))_+^2\Bigg\}.
\end{align*}
Suppose that the right hand side is negative with positive probability.
Then there exists $\varepsilon>0$ and a sequence $\tau_n\uparrow\infty$ 
of random times such that 
\begin{equation}
\label{eq:coras1}
        \sup_{d\in\mathcal{\bar D}_q}(I_{\tau_n}(d))_+^2
        - \sup_{f\in\mathcal{M}_q(4\sqrt{\log\log \tau_n/\tau_n})}
        (I_{\tau_n}(d_f))_+^2 \le -\varepsilon
\end{equation}
for all $n$ with positive probability.  We will show that this entails a 
contradiction.

By Theorem \ref{thm:clil} (which can be applied here as 
$\mathcal{N}(\mathcal{D}_q,\delta)=\mathcal{N}(\cl\mathcal{D}_q,\delta)$ 
for all $\delta>0$), the process $(I_{\tau_n})_{n\ge 0}$ is 
$\mathbf{P}^\star$-a.s.\ relatively compact in 
$\ell_\infty(\cl\mathcal{D}_q)$ with
\begin{equation}
\label{eq:coras2}
        \inf_{g\in L^2_0(f^\star d\mu)}\sup_{d\in\cl\mathcal{D}_q}
        |I_{\tau_n}(d)-\langle d,g\rangle|\xrightarrow{n\to\infty}0
        \quad\mathbf{P}^\star\mbox{\rm-a.s.}
\end{equation}
Then there is a set of positive probability on which (\ref{eq:coras1}) and
(\ref{eq:coras2}) hold simultaneously.  We now concentrate our attention 
on a single sample path in this set.  For any such path, we can clearly 
find a further subsequence $\sigma_n\uparrow\infty$ such that 
$\sup_{d\in\cl\mathcal{D}_q}|I_{\sigma_n}(d)-\langle d,g\rangle|\to 0$ 
as $n\to\infty$ for some element $g\in L^2_0(f^\star d\mu)$.  Therefore,
we obtain
\begin{multline*}
        \sup_{d\in\cl\mathcal{D}_q}
        |(I_{\sigma_n}(d))_+^2-(\langle d,g\rangle)_+^2| 
        \le 
        \sup_{d\in\cl\mathcal{D}_q}|I_{\sigma_n}(d)-\langle d,g\rangle|^2
        \\ \mbox{}+
        2\sup_{d\in\cl\mathcal{D}_q}|I_{\sigma_n}(d)-\langle d,g\rangle|
        \sup_{d\in\cl\mathcal{D}_q}|\langle d,g\rangle|
        \xrightarrow{n\to\infty}0,
\end{multline*}
where we have used the elementary estimate $|a_+^2-b_+^2| = 
|a_+-b_+|(a_++b_+) \le |a_+-b_+|(|a_+-b_+|+2b_+) \le |a-b|(|a-b|+2|b|)$
for any $a,b\in\mathbb{R}$, and the fact that
$\sup_{d\in\cl\mathcal{D}_q}|\langle d,g\rangle|\le
\sup_{d\in\cl\mathcal{D}_q}\|d\|_2\|g\|_2\le 1$.  Thus (\ref{eq:coras1})
gives
\begin{align*}
        &\liminf_{n\to\infty}\Bigg\{
        \sup_{d\in\mathcal{\bar D}_q}(\langle d,g\rangle)_+^2
        - \sup_{f\in\mathcal{M}_q(4\sqrt{\log\log \sigma_n/\sigma_n})}
        (\langle d_f,g\rangle)_+^2 
        \Bigg\} = \\
        &\liminf_{n\to\infty}\Bigg\{
        \sup_{d\in\mathcal{\bar D}_q}(I_{\sigma_n}(d))_+^2
        - \sup_{f\in\mathcal{M}_q(4\sqrt{\log\log \sigma_n/\sigma_n})}
        (I_{\sigma_n}(d_f))_+^2 
        \Bigg\} \\
        &\le -\varepsilon.
\end{align*}
But as the map $d\mapsto\langle d,g\rangle$ is continuous in $L^2(f^\star d\mu)$
and $\cl\mathcal{D}_q(4\sqrt{\log\log \sigma_n/\sigma_n})$ is compact
in $L^2(f^\star d\mu)$ (this follows from 
$\mathcal{N}(\mathcal{D}_q,\delta)<\infty$ for all $\delta>0$),
we have
\begin{multline*}
        \sup_{f\in\mathcal{M}_q(4\sqrt{\log\log \sigma_n/\sigma_n})}
        (\langle d_f,g\rangle)_+^2 
        =  
        \sup_{d\in\cl\mathcal{D}_q(4\sqrt{\log\log \sigma_n/\sigma_n})}
        (\langle d,g\rangle)_+^2 
        \xrightarrow{n\to\infty} \mbox{}\\
        \sup_{d\in\bigcap_{n\ge 0}\cl\mathcal{D}_q(4\sqrt{\log\log 
        \sigma_n/\sigma_n})}(\langle d,g\rangle)_+^2 =
        \sup_{d\in\mathcal{\bar D}_q}(\langle d,g\rangle)_+^2.
\end{multline*}
Thus we have a contradiction, completing the proof.
\end{proof}

We now obtain a converse to the previous result.

\begin{prop}
\label{prop:asympexp2}
Let $q\ge q^\star$.  Assume that
$$
        \int_0^1 \sqrt{\log\mathcal{N}(\mathcal{D}_q,u)}\,du<\infty.
$$
Moreover, suppose that $|d|\le D$ for all $d\in\mathcal{D}_q$ with
$D\in L^{2+\alpha}(f^\star d\mu)$ for some $\alpha>0$.  Then
$$
        \limsup_{n\to\infty}\Bigg\{
        \sup_{d\in\mathcal{\bar D}_q^c}(I_n(d))_+^2
        -\frac{1}{\log\log n}\Bigg\{
        \sup_{f\in\mathcal{M}_q}\ell_n(f)-
        \ell_n(f^\star)\Bigg\}\Bigg\}
$$
is nonpositive $\mathbf{P}^\star$-a.s.
\end{prop}

\begin{proof}
Suppose the result is false.  By Proposition 
\ref{prop:asympexp}, there is $\varepsilon>0$ and a sequence 
$\tau_n\uparrow\infty$ of random times so that
\begin{multline*}
        \sup_{d\in\mathcal{\bar D}_q^c}(I_{\tau_n}(d))_+^2 
        -
        \sup_{f\in\mathcal{M}_q(4\sqrt{\log\log\tau_n/\tau_n})}
        \Bigg\{
        - h(f,f^\star)^2\,\frac{2\tau_n}{\log\log\tau_n}
        \\ \mbox{}
        +2\,I_{\tau_n}(d_f)\,
        h(f,f^\star)\,\sqrt{\frac{2\tau_n}{\log\log\tau_n}}\Bigg\}
        \ge\varepsilon
\end{multline*}
for all $n$
with positive probability.  Proceeding as in the proof of Proposition
\ref{prop:asympexp1}, we can then show that there is a sequence of
times $\sigma_n\uparrow\infty$ and some $g\in L^2_0(f^\star d\mu)$
such that
\begin{multline*}
        \limsup_{n\to\infty}\Bigg\{
        \sup_{d\in\mathcal{\bar D}_q^c}(\langle d,g\rangle)_+^2
        -
        \sup_{f\in\mathcal{M}_q(4\sqrt{\log\log\sigma_n/\sigma_n})}
        \Bigg\{
        - h(f,f^\star)^2\,\frac{2\sigma_n}{\log\log\sigma_n}
        \\ \mbox{}
        +2\,\langle d_f,g\rangle\,
        h(f,f^\star)\,\sqrt{\frac{2\sigma_n}{\log\log\sigma_n}}\Bigg\}
        \Bigg\}\ge\varepsilon.
\end{multline*}
We will show that this entails a contradiction.

Let $d_0\in\mathcal{\bar D}_q$ be a continuously accessible point.
Then there exists an $\alpha_0>0$ (depending on $d_0$) and a path 
$(f_\alpha)_{\alpha\in\mbox{}]0,\alpha_0]}$ such that 
$h(f_\alpha,f^\star)=\alpha$ for all $\alpha\in\mbox{}]0,\alpha_0]$
and $d_{f_\alpha}\to d_0$ in $L^2(f^\star d\mu)$ as $\alpha\to 0$.
Now choose the sequence
$$
        \alpha_n = \{(\langle d_0,g\rangle)_+ + \sigma_n^{-1}\}
        \sqrt{\frac{\log\log\sigma_n}{2\sigma_n}}.
$$
As $(\langle d_0,g\rangle)_+ \le \|d_0\|_2\|g\|_2\le 1$,
we clearly have 
$$
        0<\alpha_n<\alpha_0 \wedge 4\sqrt{\log\log\sigma_n/\sigma_n}
$$
for all $n$ sufficiently large.  In particular, it follows that
$f_{\alpha_n}\in\mathcal{M}_q(4\sqrt{\log\log\sigma_n/\sigma_n})$,
so that
\begin{multline*}
        \sup_{f\in\mathcal{M}_q(4\sqrt{\log\log\sigma_n/\sigma_n})}
        \Bigg\{
        2\,\langle d_f,g\rangle\,
        h(f,f^\star)\,\sqrt{\frac{2\sigma_n}{\log\log\sigma_n}}
        - h(f,f^\star)^2\,\frac{2\sigma_n}{\log\log\sigma_n}
        \Bigg\}
        \\ \mbox{}\ge
        2\,\langle d_{f_{\alpha_n}},g\rangle\,\{(\langle 
        d_0,g\rangle)_+ + \sigma_{n}^{-1}\}
        - \{(\langle d_0,g\rangle)_+ + \sigma_n^{-1}\}^2.
\end{multline*}
Therefore, we have
\begin{align*}
        &\limsup_{n\to\infty}\Bigg\{
        \sup_{d\in\mathcal{\bar D}_q^c}(\langle d,g\rangle)_+^2
        -
        \sup_{f\in\mathcal{M}_q(4\sqrt{\log\log\sigma_n/\sigma_n})}
        \Bigg\{
        - h(f,f^\star)^2\,\frac{2\sigma_n}{\log\log\sigma_n}
        \\ &\qquad\qquad\qquad\mbox{}
        +2\,\langle d_f,g\rangle\,
        h(f,f^\star)\,\sqrt{\frac{2\sigma_n}{\log\log\sigma_n}}\Bigg\}
        \Bigg\} 
	\\ &\quad\mbox{} \le
        \sup_{d\in\mathcal{\bar D}_q^c}(\langle d,g\rangle)_+^2
        - (\langle d_0,g\rangle)_+^2
\end{align*}
for any continuously accessible element $d_0\in\mathcal{\bar D}_q$.
But clearly we can choose $d_0$ to make the right hand side of this 
expression arbitrarily small.  Thus we have the desired contradiction.
\end{proof}

We can now complete the proof of Theorem \ref{thm:lowerlil}.

\begin{proof}[Proof of Theorem \ref{thm:lowerlil}]
We obtain separately the lower and upper bounds.

\textbf{Lower bound.} By Propositions \ref{prop:asympexp1} and
\ref{prop:asympexp2}, we have
\begin{multline*}
        \limsup_{n\to\infty}\frac{1}{\log\log n}\left\{
        \sup_{f\in\mathcal{M}_q}\ell_n(f)-
        \sup_{f\in\mathcal{M}_p}\ell_n(f)\right\} \ge \mbox{}\\
        \limsup_{n\to\infty}\left\{
        \sup_{d\in\mathcal{\bar D}_q^c}(I_n(d))_+^2-
        \sup_{d\in\mathcal{\bar D}_p}(I_n(d))_+^2
        \right\}\quad\mathbf{P}^\star\mbox{\rm-a.s.}
\end{multline*}
Now fix any $g\in L^2_0(f^\star d\mu)$.  By Theorem \ref{thm:clil} (which 
applies here as $\mathcal{N}(\mathcal{D}_q,\delta)= 
\mathcal{N}(\cl\mathcal{D}_q,\delta)\ge \mathcal{N}(\mathcal{\bar 
D}_q,\delta)$ for all $\delta>0$), there is a sequence 
$\tau_n\uparrow\infty$ of random times such that 
$I_{\tau_n}\to\langle\,\cdot\,,g\rangle$ in $\ell_\infty(\mathcal{\bar 
D}_q)$ $\mathbf{P}^\star$-a.s.  Therefore
$$
        \sup_{d\in\mathcal{\bar D}_q^c}(I_{\tau_n}(d))_+^2-
        \sup_{d\in\mathcal{\bar D}_p}(I_{\tau_n}(d))_+^2
        \xrightarrow{n\to\infty} 
        \sup_{d\in\mathcal{\bar D}_q^c}(\langle d,g\rangle)_+^2-
        \sup_{d\in\mathcal{\bar D}_p}(\langle d,g\rangle)_+^2
        \quad\mathbf{P}^\star\mbox{\rm-a.s.},
$$
so that certainly
$$
        \limsup_{n\to\infty}\frac{1}{\log\log n}\left\{
        \sup_{f\in\mathcal{M}_q}\ell_n(f)-
        \sup_{f\in\mathcal{M}_p}\ell_n(f)\right\} \ge 
        \sup_{d\in\mathcal{\bar D}_q^c}(\langle d,g\rangle)_+^2-
        \sup_{d\in\mathcal{\bar D}_p}(\langle d,g\rangle)_+^2
$$
$\mathbf{P}^\star$-a.s.  But as this inequality holds for every $g\in 
L^2_0(f^\star d\mu)$, taking the supremum over $g$ gives the requisite 
lower bound.

\textbf{Upper bound.} By Propositions \ref{prop:asympexp1} and 
\ref{prop:asympexp2}, we have
\begin{multline*}
        \limsup_{n\to\infty}\frac{1}{\log\log n}\left\{
        \sup_{f\in\mathcal{M}_q}\ell_n(f)-
        \sup_{f\in\mathcal{M}_p}\ell_n(f)\right\} \le \mbox{}\\
        \limsup_{n\to\infty}\left\{
        \sup_{d\in\mathcal{\bar D}_q}(I_n(d))_+^2-
        \sup_{d\in\mathcal{\bar D}_p^c}(I_n(d))_+^2
        \right\}\quad\mathbf{P}^\star\mbox{\rm-a.s.}
\end{multline*}
It is elementary that for any $d,d'\in\mathcal{\bar D}_q$ and
$g\in L^2_0(f^\star d\mu)$
\begin{align*}
        &(I_n(d))_+^2-(I_n(d'))_+^2 \\
        &\quad\le
        |(I_n(d))_+^2-(\langle d,g\rangle)_+^2|
        +|(I_n(d'))_+^2-(\langle d',g\rangle)_+^2| 
        +(\langle d,g\rangle)_+^2-(\langle d',g\rangle)_+^2 \\
        &\quad\le
        2\sup_{d\in\mathcal{\bar D}_q}|(I_n(d))_+^2-(\langle d,g\rangle)_+^2|
        +(\langle d,g\rangle)_+^2-(\langle d',g\rangle)_+^2.
\end{align*}
Taking the supremum over $d\in\mathcal{\bar D}_q$ and the infimum over
$d'\in\mathcal{\bar D}_p^c$, we find that
\begin{align*}
        &\sup_{d\in\mathcal{\bar D}_q}(I_n(d))_+^2
        -\sup_{d\in\mathcal{\bar D}_p^c}(I_n(d))_+^2 
        \\
        &\quad\le
        2\sup_{d\in\mathcal{\bar D}_q}|(I_n(d))_+^2-(\langle d,g\rangle)_+^2|
        +\sup_{d\in\mathcal{\bar D}_q}(\langle d,g\rangle)_+^2
        -\sup_{d\in\mathcal{\bar D}_p^c}(\langle d,g\rangle)_+^2 \\
        &\quad\le
        2\sup_{d\in\mathcal{\bar D}_q}|(I_n(d))_+^2-(\langle d,g\rangle)_+^2|
        +\sup_{g\in L^2_0(f^\star d\mu)}\left\{
        \sup_{d\in\mathcal{\bar D}_q}(\langle d,g\rangle)_+^2
        -\sup_{d\in\mathcal{\bar D}_p^c}(\langle d,g\rangle)_+^2\right\}.
\end{align*}
But as this holds for any $g\in L^2_0(f^\star d\mu)$, we finally obtain
\begin{multline*}
        \sup_{d\in\mathcal{\bar D}_q}(I_n(d))_+^2
        -\sup_{d\in\mathcal{\bar D}_p^c}(I_n(d))_+^2 
	\le
        2\inf_{g\in L^2_0(f^\star d\mu)}
        \sup_{d\in\mathcal{\bar D}_q}|(I_n(d))_+^2-(\langle d,g\rangle)_+^2|
        \\ \mbox{}
        +\sup_{g\in L^2_0(f^\star d\mu)}\left\{
        \sup_{d\in\mathcal{\bar D}_q}(\langle d,g\rangle)_+^2
        -\sup_{d\in\mathcal{\bar D}_p^c}(\langle d,g\rangle)_+^2\right\}.
\end{multline*}
It follows as in the proof of Proposition \ref{prop:asympexp1} that the 
first term in this expression converges to zero $\mathbf{P}^\star$-a.s.
The requisite upper bound follows immediately.
\end{proof}

Finally, we now complete the proof of Corollary \ref{cor:lowerlil}

\begin{proof}[Proof of Corollary \ref{cor:lowerlil}]
It evidently suffices to prove that
\begin{equation}
\label{eq:inconpf}
        \Gamma:=
        \sup_{g\in L^2_0(f^\star d\mu)}\Bigg\{
        \sup_{d\in\mathcal{\bar D}_q^c}(\langle d,g\rangle)_+^2
        -\sup_{d\in\mathcal{\bar D}_{q^\star}}(\langle d,g\rangle)_+^2
        \Bigg\}>0.
\end{equation}
To this end, note that by direct computation
$$
        \langle 1,d_f\rangle = \frac{\int \sqrt{ff^\star}\,d\mu-1}{h(f,f^\star)}
	= -\frac{h(f,f^\star)}{2}.
$$
Choose $(f_n)_{n\ge 0}\subset\mathcal{M}_q\backslash\{f^\star\}$
such that $h(f_n,f^\star)\to 0$ and $d_{f_n}\to d_0\in\mathcal{\bar D}_q$, 
then
$$
        \langle 1,d_0\rangle =
	\lim_{n\to\infty}\langle 1,d_{f_n}\rangle =
        -\lim_{n\to\infty}\frac{h(f_n,f^\star)}{2}=0.
$$
Moreover, it is immediate that $\|d_0\|_2\le 1$.  We have therefore
shown that $\mathcal{\bar D}_q\subset L^2_0(f^\star d\mu)$.  Now choose
$g\in\mathcal{\bar D}_q^c\backslash\mathcal{\bar D}_{q^\star}$.  
As $\mathcal{\bar D}_{q^\star}$ is closed, it follows directly that
$$
        \sup_{d\in\mathcal{\bar D}_q^c}(\langle d,g\rangle)_+^2=1,\qquad
        \qquad
        \sup_{d\in\mathcal{\bar D}_{q^\star}}(\langle d,g\rangle)_+^2<1.
$$
Therefore (\ref{eq:inconpf}) holds, and the proof is complete.
\end{proof}

\section{Proof of Theorem \ref{thm:mainconsist}}
\label{sec:proofmainconsist}

The proof of Theorem \ref{thm:mainconsist} is based on Theorem 
\ref{thm:upperlil} and the following result.

\begin{prop}
\label{prop:liklln}
Let $\mathcal{M}^n$ for $n\ge 1$ be a family of strictly positive 
probability densities with respect to a reference measure $\mu$
such that $\mathcal{M}^n\subseteq\mathcal{M}^{n+1}$ for all $n$.
Define $\mathcal{M}=\bigcup_n\mathcal{M}^n$, and let $f^\star$ be another
probability density with respect to $\mu$ such that 
$f^\star\not\in\cl\mathcal{M}$, where $\cl\mathcal{M}$ denotes the
$L^1(d\mu)$-closure of $\mathcal{M}$.  Let $\mathcal{H}^n = 
\{\sqrt{f/f^\star}:f\in\mathcal{M}^n\}$, and
suppose there exist $K(n)\ge 1$ and $p\ge 1$ so that
$$
        \mathcal{N}(\mathcal{H}^n,\delta)
        \le\left(
        \frac{K(n)}{\delta}
        \right)^p
$$
for all $\delta\le 1$ and $n\ge 1$, where $\mathcal{N}(\mathcal{H}^n,\delta)$
is the minimal number of brackets of $L^2(f^\star d\mu)$-width $\delta$
needed to cover $\mathcal{H}^n$.
Let $(X_i)_{i\in\mathbb{N}}$ be i.i.d.\ with distribution $f^\star d\mu$.
If in addition $\log K(n)=o(n)$, then we have
$$
        \limsup_{n\to\infty}\sup_{f\in\mathcal{M}^n}
        \frac{1}{n}\sum_{j=1}^n\log\left(
        \frac{f(X_j)}{f^\star(X_j)}
        \right)<0\quad\mbox{\rm a.s.}
$$
\end{prop}

\begin{proof}
As in the proof of Theorem \ref{thm:deviationbound}, we have
$$
        \frac{1}{n}\sum_{j=1}^n\log\left(\frac{f(X_j)}{f^\star(X_j)}\right)
        \le
        4n^{-1/2} \nu_n(\log(\{\bar f/f^\star\}^{1/2})) - 2 D(f^\star||\bar f).
$$
The following claim will be proved below:
$$
        \lim_{n\to\infty}
        \sup_{f\in\mathcal{M}^n}n^{-1/2} \nu_n(\log(\{\bar f/f^\star\}^{1/2}))
        =0\quad\mbox{\rm a.s.}
$$
Using the claim, the proof is easily completed: indeed, if the claim holds,
then we have a.s.
$$
        \limsup_{n\to\infty}\sup_{f\in\mathcal{M}^n}
        \frac{1}{n}\sum_{j=1}^n\log\left(
        \frac{f(X_j)}{f^\star(X_j)}
        \right)
        \le
        -2\inf_{f\in\mathcal{M}} D(f^\star||\bar f)<0
$$
where the last inequality follows from Pinsker's inequality and
$f^\star\not\in\cl\mathcal{M}$.

It therefore remains to prove the claim.  To this end we apply
\cite[Theorem 5.11]{vdG00} as in the proof of \cite[Theorem 7.4]{vdG00}
(cf.\ Theorem \ref{thm:deviationbound} above),
which yields
$$
        \mathbf{P}\left[
        \sup_{f\in\mathcal{M}^n}|n^{-1/2} \nu_n(\log(\{\bar f/f^\star\}^{1/2}))|
        \ge\alpha
        \right]\le
        C\,e^{-n\alpha^2/C}
$$
for every $\alpha>0$ such that
$C\sqrt{p}\,(1+\sqrt{\log K(n)})\le\alpha\sqrt{n}\le 32\sqrt{n}$
and $n\ge 1$, where $C$ is a universal constant.  As $\log K(n)=o(n)$,
we have
$$
        \sum_{n\ge 1}
        \mathbf{P}\left[
        \sup_{f\in\mathcal{M}^n}|n^{-1/2} \nu_n(\log(\{\bar f/f^\star\}^{1/2}))|
        \ge\alpha
        \right]<\infty
$$
for $0<\alpha\le 32$, so the claim follows from Borel-Cantelli.
\end{proof}

We can now complete the proof of Theorem \ref{thm:mainconsist}.

\begin{proof}[Proof of Theorem \ref{thm:mainconsist}]
Define
$$
	\Delta_n(q,q^\star) = 
        \sup_{f\in\mathcal{M}_q^n}
        \ell_n(f)-\sup_{f\in \mathcal{M}_{q^\star}^n}\ell_n(f).
$$
By Theorem \ref{thm:upperlil} and easy manipulations, 
$\mathbf{P}^\star$-a.s.
\begin{align*}
        &\limsup_{n\to\infty}
        \sup_{q>q^\star}
        \frac{1}{\pen(n,q)-\pen(n,q^\star)}\Delta_n(q,q^\star) \\
        &\mbox{}\le
        \lim_{n\to\infty}
        \sup_{q>q^\star}
        \frac{\eta(q)\{\log K(2n)\vee \log\log n\}}{\pen(n,q)-
        \pen(n,q^\star)} \times\mbox{} \\ &
        \quad\limsup_{n\to\infty}
        \frac{1}{\log K(2n)\vee \log\log n}
        \sup_{q>q^\star}
        \frac{1}{\eta(q)}\Delta_n(q,q^\star)=0.
\end{align*}
Therefore, $\mathbf{P}^\star$-a.s.\ eventually as $n\to\infty$
$$
        \sup_{f\in\mathcal{M}_q^n}\ell_n(f) - \pen(n,q)
        < \sup_{f\in\mathcal{M}_{q^\star}^n}\ell_n(f)
        - \pen(n,q^\star)
$$
for all $q>q^\star$.  It follows that $\limsup_{n\to\infty}\hat q_n\le 
q^\star$ $\mathbf{P}^\star$-a.s., that is, the penalized likelihood order 
estimator does not asymptotically overestimate the order.

On the other hand, we note that for every $q<q^\star$
$$
        \limsup_{n\to\infty}\frac{1}{n}\left\{
        \sup_{f\in\mathcal{M}_q^n}\ell_n(f)-
        \sup_{f\in\mathcal{M}_{q^\star}^n}\ell_n(f)
        \right\} 
        \le
        \limsup_{n\to\infty}
        \sup_{f\in\mathcal{M}_q^n}
        \frac{1}{n}\sum_{j=1}^n \log\left(
        \frac{f(X_j)}{f^\star(X_j)}\right)
$$
which is strictly negative $\mathbf{P}^\star$-a.s.\ by Proposition 
\ref{prop:liklln}, where we have used that $\log K(n)=o(n)$ and that 
$\mathcal{N}(\mathcal{H}_q^n(2),\delta)\le 
\mathcal{N}(\mathcal{H}_{q^\star}^n(2),\delta) \le 
(2K(n)/\delta)^{\eta(q^\star)}$ for all $\delta\le 2$ and $n$ 
sufficiently large. As $\pen(n,q)/n\to 0$ as $n\to\infty$ for
$q<q^\star$
$$
        \limsup_{n\to\infty}\max_{q<q^\star}\frac{1}{n}\left\{
	\Delta_n(q,q^\star)
	-\pen(n,q)+\pen(n,q^\star)
        \right\} < 0
$$
$\mathbf{P}^\star$-a.s.  In particular, we find that
$\mathbf{P}^\star$-a.s.\ eventually as $n\to\infty$
$$
        \sup_{f\in\mathcal{M}_q^n}\ell_n(f) - \pen(n,q)
        < \sup_{f\in\mathcal{M}_{q^\star}^n}\ell_n(f)
        - \pen(n,q^\star)
$$
for all $q<q^\star$.  It follows that $\liminf_{n\to\infty}\hat q_n\ge
q^\star$ $\mathbf{P}^\star$-a.s., that is, the penalized likelihood order 
estimator does not asymptotically underestimate the order.
\end{proof}

Finally, let us prove Corollary \ref{cor:inconsist}.

\begin{proof}[Proof of Corollary \ref{cor:inconsist}]
It is shown in the proof of Corollary \ref{cor:lowerlil} that
$$
        \Gamma:=
        \sup_{g\in L^2_0(f^\star d\mu)}\Bigg\{
        \sup_{d\in\mathcal{\bar D}_q^c}(\langle d,g\rangle)_+^2
        -\sup_{d\in\mathcal{\bar D}_{q^\star}}(\langle d,g\rangle)_+^2
        \Bigg\}>0.
$$
By Theorem \ref{thm:lowerlil}, we have $\mathbf{P}^\star$-a.s.
\begin{align*}
        &\limsup_{n\to\infty}
        \frac{1}{\pen(n,q)-\pen(n,q^\star)}
        \Bigg\{\sup_{f\in\mathcal{M}_q}\ell_n(f)-
        \sup_{f\in\mathcal{M}_{q^\star}}\ell_n(f)\Bigg\} \\
        &\mbox{}\ge
        \frac{1}{C\{\eta(q)-\eta(q^\star)\}}
        \sup_{g\in L^2_0(f^\star d\mu)}\Bigg\{
        \sup_{d\in\mathcal{\bar D}_q^c}(\langle d,g\rangle)_+^2
        -\sup_{d\in\mathcal{\bar D}_{q^\star}}(\langle d,g\rangle)_+^2
        \Bigg\}.
\end{align*}
Therefore, choosing 
$C<\Gamma/\{\eta(q)-\eta(q^\star)\}$, we find that
$$
        \sup_{f\in\mathcal{M}_q}\ell_n(f)-\pen(n,q) >
        \sup_{f\in\mathcal{M}_{q^\star}}\ell_n(f) -\pen(n,q^\star)
$$
infinitely often $\mathbf{P}^\star$-a.s., so 
$\hat q_n\ne q^\star$ infinitely often $\mathbf{P}^\star$-a.s.
\end{proof}

\section{Proof of Proposition \ref{prop:mixconsist}}
\label{sec:proofmixconsist}

The proofs of consistency in Propositions \ref{prop:mixconsist} and 
\ref{prop:mixconsist2} follow almost immediately from Theorem 
\ref{thm:mainconsist}, Theorem \ref{cor:localmixtures}, and Example 
\ref{ex:gaussian}.  Let us begin with Proposition 
\ref{prop:mixconsist2}.

\begin{proof}[Proof of Proposition \ref{prop:mixconsist2}]
By Example \ref{ex:gaussian}, the assumption
of Theorem \ref{thm:mainconsist} holds with $\eta(q)=
18(d+1)q+1$ and $\log K(n) = \log C_1^\star+ C_2^\star T(n)^2$.
The desired consistency results now follow immediately from
Theorem \ref{thm:mainconsist}.
\end{proof}

The consistency part of Proposition \ref{prop:mixconsist} follows 
similarly.  The main difficulty here is to establish the condition 
$\mathcal{\bar D}_q^c\backslash\mathcal{\bar D}_{q^\star}\ne 
\varnothing$ of Corollary \ref{cor:inconsist}, which is needed to prove 
the inconsistency part of Proposition \ref{prop:mixconsist}. In the 
proof of the latter condition, we rely on the geometric results 
on mixtures established in \cite{GasHanGeo}.  In the remainder of this 
section, we always assume that we are in the setting of Proposition 
\ref{prop:mixconsist}.

\begin{lem}
\label{lem:dqstar}
Suppose that Assumption A holds.  Then we have
$$
        \mathcal{\bar D}_{q^\star} =
        \Bigg\{\frac{L}{\|L\|_2}:
        L=\sum_{i=1}^{q^\star}
        \Bigg\{\eta_i\,\frac{f_{\theta_i^\star}}{f^\star}
        +\beta_i^*\,\frac{D_1f_{\theta_i^\star}}{f^\star}
        \Bigg\},~
        \eta_i\in\mathbb{R}, ~
	\beta_i\in\mathbb{R}^{d}, ~
        \sum_{i=1}^{q^\star}\eta_i=0
        \Bigg\}.
$$
\end{lem}

\begin{proof}
Let $(f_n)_{n\ge 1}\subset\mathcal{M}_{q^\star}$ be such that
$h(f_n,f^\star)\to 0$ and $d_{f_n}\to d_0\in\mathcal{\bar D}_{q^\star}$.  
By \cite[Theorem 3.7]{GasHanGeo}, 
we may assume without loss of generality 
that $f_n=\sum_{i=1}^{q^\star}\pi_i^nf_{\theta_i^n}$ with 
$\theta_i^n\to\theta_i^\star$ and $\pi_i^n\to\pi_i^\star$ for 
$i=1,\ldots,q^\star$.  Taylor expansion gives
$$
        \frac{f_n-f^\star}{f^\star} = L_n + R_n,\quad
        \quad |R_n| \le 
        \frac{d}{2}\,H_2\sum_{i=1}^{q^\star}
        \pi_i^n\|\theta_i^n-\theta_i^\star\|^2,
$$
where
$$
        L_n =
        \sum_{i=1}^{q^\star}\left\{
        (\pi_i^n-\pi_i^\star)\,\frac{f_{\theta_i^\star}}{f^\star}
        +\pi_i^n(\theta_i^n-\theta_i^\star)^*\,
        \frac{D_1f_{\theta_i^\star}}{f^\star}
        \right\}.
$$
Proceeding as in 
\cite[Lemma 3.12--3.13]{GasHanGeo}, we can estimate
$$
        \left\|d_{f_n}-
        \frac{L_n}{\|L_n\|_2}\right\|_2 \le
        2\|S\|_4^2\{2\|S\|_2+1\}\,h(f_n,f^\star)+
        \{\|S\|_2+1\}\,\frac{\|R_n\|_2}{\|L_n\|_2}.
$$
But using \cite[Theorem 3.7]{GasHanGeo},
we have for $n$ sufficiently large
$$
        \|L_n\|_2 \ge \|L_n\|_1 \ge c^\star
        \sum_{i=1}^{q^\star}\pi_i^n\|\theta_i^n-\theta_i^\star\|.
$$
Thus we have
$$
        \frac{\|R_n\|_2}{\|L_n\|_2} \le
        \frac{d\|H_2\|_2}{2c^\star}
        \frac{\sum_{i=1}^{q^\star}
        \pi_i^n\|\theta_i^n-\theta_i^\star\|^2}{
        \sum_{i=1}^{q^\star}
        \pi_i^n\|\theta_i^n-\theta_i^\star\|}
        \le  \frac{d\|H_2\|_2}{2c^\star}\max_{i=1,\ldots,q^\star}
        \|\theta_i^n-\theta_i^\star\|\xrightarrow{n\to\infty}0.
$$
Therefore $L_n/\|L_n\|_2\to d_0$ in 
$L^2(f^\star d\mu)$.  Now define
\begin{eqnarray*}
        \eta_i^n=\frac{\pi_i^n-\pi_i^\star}{Z_n},\quad
        \beta_i^n=\frac{\pi_i^n(\theta_i^n-\theta_i^\star)}{Z_n},
        \\
        Z_n=\sum_{i=1}^{q^\star}
        \{|\pi_i^n-\pi_i^\star|+\|\pi_i^n(\theta_i^n-\theta_i^\star)\|\}.
\end{eqnarray*}
As $\sum_{i=1}^{q^\star}\{|\eta_i^n|+\|\beta_i^n\|\}=1$ for all $n$,
we may extract a subsequence such that $\eta_i^n\to\eta_i$,
$\beta_i^n\to\beta_i$, and $\sum_{i=1}^{q^\star}\{|\eta_i|+\|\beta_i\|\}=1$. 
We obtain immediately
$$
        d_0 = \frac{L}{\|L\|_2},\qquad
        L=
        \sum_{i=1}^{q^\star}\left\{
        \eta_i\,\frac{f_{\theta_i^\star}}{f^\star}
        +\beta_i^*
        \frac{D_1f_{\theta_i^\star}}{f^\star}\right\}.
$$
Clearly $\sum_{i=1}^{q^\star}\eta_i=0$.  Thus we have shown that
any $d_0\in\mathcal{\bar D}_{q^\star}$ has the desired form.

It remains to show that any function of the desired form is in fact an 
element of $\mathcal{\bar D}_{q^\star}$.  To this end, fix
$\eta_i\in\mathbb{R}$, $\beta_i\in\mathbb{R}^{d}$ with 
$\sum_{i=1}^{q^\star}\eta_i=0$, 
and define $f_t$ for $t>0$ as
$$
        f_t = 
        \sum_{i=1}^{q^\star}
        (\pi_i^\star+t\eta_i)\,
        f_{\theta_i^\star+\beta_it/\pi_i^\star}.
$$
Clearly $f_t\in\mathcal{M}_{q^\star}$ for all $t$ sufficiently small, and
$f_t\to f^\star$ as $t\to 0$.  But
$$
        \frac{f_t-f^\star}{t} =
        \sum_{i=1}^{q^\star}
        \pi_i^\star\,\frac{f_{\theta_i^\star+\beta_it/\pi_i^\star}
        -f_{\theta_i^\star}}{t}
        + \sum_{i=1}^{q^\star}
        \eta_i\,f_{\theta_i^\star+\beta_it/\pi_i^\star}.
$$
Therefore clearly
$$
        \frac{1}{t}\frac{f_t-f^\star}{f^\star}
        \xrightarrow{t\to 0}
        \sum_{i=1}^{q^\star}\left\{
        \eta_i\,\frac{f_{\theta_i^\star}}{f^\star}
        +\beta_i^*
        \frac{D_1f_{\theta_i^\star}}{f^\star}\right\} = L.
$$
Using \cite[Lemma 3.12]{GasHanGeo},
we obtain
$$
        \lim_{t\to 0}d_{f_t} = \lim_{t\to 0}
        \frac{(f_t-f^\star)/tf^\star}{
        \|(f_t-f^\star)/tf^\star\|_2} =
        \frac{L}{\|L\|_2}.
$$
Thus any function of the desired form is in $\mathcal{\bar D}_{q^\star}$.
\end{proof}

\begin{rem}
The above proof in fact shows that
$\mathcal{\bar D}_{q^\star}=\mathcal{\bar D}_{q^\star}^c$.
\end{rem}

We can now complete the proof of Proposition \ref{prop:mixconsist}.

\begin{proof}[Proof of Proposition \ref{prop:mixconsist}]
We first prove consistency of the penalty $\pen(n,q)=q\,\omega(n)$.
Note that by Theorem \ref{cor:localmixtures}, the assumption
of Corollary \ref{cor:mainconsist} holds with $\eta(q)=
18(d+1)q+1\le 19(d+1)q$.  Thus consistency of $\pen(n,q)=q\,\omega(n)$
follows directly from Corollary \ref{cor:mainconsist} using
$\varpi(n) = \omega(n)/19(d+1)$.

To prove that
the penalty $\pen(n,q)=C\,q\,\log\log n$ is inconsistent for 
$C>0$ sufficiently small, it suffices to show that $\mathcal{\bar 
D}_{q^\star+1}^c\backslash\mathcal{\bar D}_{q^\star}$ is nonempty.  
Indeed, if this is the case then we can apply Corollary 
\ref{cor:inconsist} with $q=q^\star+1$, where the requisite entropy 
assumption follows from Theorem \ref{thm:mainconsist}.

Fix $v\in\mathbb{R}^d$, and consider $f_t$ defined for 
$t>0$ as follows:
$$
        f_t = \frac{\pi_1^\star}{2}\,(f_{\theta_1^\star+vt}+
        f_{\theta_1^\star-vt})+
        \sum_{i=2}^{q^\star}\pi_i^\star f_{\theta_i^\star}.
$$
Clearly $f_t\in\mathcal{M}_{q^\star+1}$ for all $t$ sufficiently small,
$f_t\to f^\star$ as $t\to 0$, and 
$$
        \frac{f_t-f^\star}{t^2} = \frac{\pi_1^\star}{2}\,
        \frac{f_{\theta_1^\star+vt}-2\,f_{\theta_1^\star}
        +f_{\theta_1^\star-vt}}{t^2}
        \xrightarrow{t\to 0}
        \frac{\pi_1^\star}{2}\,v^*D_2f_{\theta_1^\star}v.
$$
As in the proof of Lemma \ref{lem:dqstar}, we find that
$$
        \lim_{t\to 0}d_{f_t} = \lim_{t\to 0}
        \frac{(f_t-f^\star)/t^2f^\star}{
        \|(f_t-f^\star)/t^2f^\star\|_2} =
        \frac{v^*D_2f_{\theta_1^\star}v}{\|v^*D_2f_{\theta_1^\star}v\|_2} 
	= d_0.
$$
By construction, $d_0\in\mathcal{\bar D}_{q^\star+1}^c$.  But
by \cite[Theorem 3.7]{GasHanGeo}, 
the functions $f_{\theta_i^\star}$,
$D_1f_{\theta_i^\star}$, and $v^*D_2f_{\theta_i^\star}v$ 
($i=1,\ldots,q^\star$) are all linearly independent.  Together with Lemma 
\ref{lem:dqstar}, this shows that $d_0\not\in\mathcal{\bar D}_{q^\star}$.  
Thus $d_0\in\mathcal{\bar D}_{q^\star+1}^c\backslash\mathcal{\bar 
D}_{q^\star}$.
\end{proof}

\vskip.2cm

\textbf{Acknowledgment.}
The authors would like to thank Michel Ledoux for suggesting some helpful references.

%
\bibliographystyle{IEEEtran}


%





\end{document}